\documentclass[review,onefignum,onetabnum]{siamart190516}

\usepackage{tikz}
\usepackage{multicol}

\usetikzlibrary{calc,patterns,decorations.pathreplacing,calligraphy}
\tikzstyle{bag} = [align=center]

\usepackage{amsfonts}
\usepackage{graphicx}
\usepackage{epstopdf}
\usepackage{algorithmic}
\ifpdf
  \DeclareGraphicsExtensions{.eps,.pdf,.png,.jpg}
\else
  \DeclareGraphicsExtensions{.eps}
\fi

\usepackage{enumitem}
\setlist[enumerate]{leftmargin=.5in}
\setlist[itemize]{leftmargin=.5in}

\newsiamremark{remark}{Remark}
\newsiamremark{hypothesis}{Hypothesis}
\crefname{hypothesis}{Hypothesis}{Hypotheses}
\newsiamthm{claim}{Claim}

\headers{Efficient OUU using Derivative-Informed Neural Operators}{D. Luo, T. O'Leary-Roseberry, P. Chen, O. Ghattas}

\title{
Efficient PDE-constrained optimization under high-dimensional uncertainty using Derivative-Informed Neural Operators
\thanks{Submitted to the editors DATE.
\funding{This research partially
     supported by DOE grants DE-SC0019303 and DE-SC0023171;
     NSF DMS grant 2012453; and DOD MURI FA9550-21-1-0084.}}
}

\author{Dingcheng Luo\thanks{Oden Institute for Computational Engineering and Sciences, The University of Texas at Austin, Austin, TX, USA
(\email{dc.luo@utexas.edu}).
}
\and Thomas O'Leary-Roseberry\footnotemark[2]
\and \linebreak Peng Chen\thanks{
  School of Computational Science and Engineering, Georgia Institute of Technology, Atlanta, GA, USA
}
\and Omar Ghattas\footnotemark[2] \thanks{
  Walker Department of Mechanical Engineering, The University of Texas at Austin, Austin, TX, USA
}
}

\usepackage{amsopn}

\newcommand{\cA}{\mathcal{A}}
\newcommand{\cB}{\mathcal{B}}
\newcommand{\cC}{\mathcal{C}}

\newcommand{\cJ}{\mathcal{J}}

\newcommand{\cL}{\mathcal{L}}
\newcommand{\cM}{\mathcal{M}}
\newcommand{\cN}{\mathcal{N}}

\newcommand{\cP}{\mathcal{P}}

\newcommand{\cU}{\mathcal{U}}

\newcommand{\cX}{\mathcal{X}}
\newcommand{\cY}{\mathcal{Y}}
\newcommand{\cZ}{\mathcal{Z}}

\newcommand{\bE}{\mathbb{E}}

\newcommand{\bR}{\mathbb{R}}

\newcommand{\ad}{\mathrm{ad}}

\newcommand{\st}{\text{subject to }}

\newcommand{\CVaR}{\mathrm{CVaR}}

\newcommand{\NN}{\mathrm{NN}}

\newcommand{\PDE}{\mathrm{PDE}}

\newcommand{\HS}{\mathrm{HS}}
\newcommand{\symm}{\mathrm{symm}}

\ifpdf
\hypersetup{
  pdftitle={Efficient PDE-constrained optimization under high-dimensional uncertainty using Derivative-Informed Neural Operators}
  pdfauthor={D. Luo, T. O'Leary-Roseberry, P. Chen, O.Ghattas}
}
\fi

\usepackage{subcaption}

\nolinenumbers

\begin{document}

\maketitle

\begin{abstract}
We propose a novel machine learning framework for solving optimization problems governed by large-scale partial differential equations (PDEs) with high-dimensional random parameters. 
Such optimization under uncertainty (OUU) problems may be computational prohibitive using classical methods, particularly when a large number of samples is needed to evaluate risk measures 
at every iteration of an optimization algorithm, where each sample requires the solution of an expensive-to-solve PDE. 
To address this challenge, we propose a new neural operator approximation of the PDE solution operator that has the combined merits of (1) accurate approximation of not only the map from the joint inputs of random parameters and optimization variables to the PDE state, but also its derivative with respect to the optimization variables, (2) efficient construction of the neural network using reduced basis architectures that are scalable to high-dimensional OUU problems, and (3) requiring only a limited number of training data to achieve high accuracy for both the PDE solution and the OUU solution. 
We refer to such neural operators as multi-input reduced basis derivative informed neural operators (MR-DINOs).
We demonstrate the accuracy and efficiency our approach through several numerical experiments, i.e. the risk-averse control of a semilinear elliptic PDE and the steady state Navier--Stokes equations in two and three spatial dimensions, each involving random field inputs.
Across the examples, MR-DINOs offer $10^{3}$---$10^{7} \times$ reductions in execution time, and are able to produce OUU solutions of comparable accuracies to those from standard PDE based solutions while being over $10 \times$ more cost-efficient after factoring in the cost of construction. 

\end{abstract}




\section{Introduction}\label{section:intro}

PDE-constrained optimization problems arise in many computational science and engineering fields. 
Canonical examples of such problems include optimal design, where the goal is to find the best system configuration given constraints, and optimal control, where the aim is to determine the optimal operation of a system while adhering to constraints.
In real-world applications, uncertainties are inevitable and arise from many sources in PDE models, e.g., PDE coefficients that parametrize the system properties, initial and boundary conditions, source terms, and computational geometries.
In order to achieve robustness of optimal solutions, it is crucial to account for these uncertainties in solving PDE-constrained optimization problems.
In such PDE-constrained optimization under uncertainty (OUU) problems, the uncertainty is modeled by a probability distribution, and the optimization objective is formulated using risk measures of a performance function, which can often be written as the integral of a scalar quantity over the probability distribution.

Solution of PDE-constrained OUU problems is challenging for the following reasons. 
(1) Each optimization iteration requires solving numerous PDEs to estimate the risk measure objective, e.g. by sample average approximation (SAA). 
Accurately estimating optimization risk measures may require a large number of samples, especially for risk measures that focus on tail probabilities and rare events. 
(2) To scale up to high-dimensional optimization variables, methods that compute or approximate the Hessian are required, leading to the need for additional linearized (adjoint) PDEs to be solved at each sample.
(3) The space of the random parameters may be high-dimensional, or even infinite-dimensional, 
 which precludes the use of deterministic quadrature methods, which suffer from the curse of dimensionality. 
As a result of these computational challenges, solving PDE-constrained OUU problems for complex systems, such as large-scaled, multiphysics, or multiscale systems, is often computationally prohibitive using traditional PDE solver-based methods. Recent developments to partially address these challenges include methods such as multigrid \cite{Borzi10}, multifidelity \cite{NgWillcox14}, multilevel \cite{AliUllmannHinze17, ChenQuarteroniRozza16}, stochastic Galerkin \cite{KunothSchwab16}, stochastic collocation \cite{KouriHeinkenschloosVanBloemenWaanders12,TieslerKirbyXiuEtAl12}, Taylor approximation \cite{AlexanderianPetraStadlerEtAl17,  ChenGhattas21, ChenHabermanGhattas21,ChenVillaGhattas19}, model reduction \cite{AllaHinzeKolvenbachEtAl19, ChenQuarteroni14, LassUlbrich17, ZahrCarlbergKouri19}, and neural networks \cite{EigelHaaseNeumann22,GuthSchillingsWeissmann21}. Despite the progress, the computational challenges remain formidable, particularly for OUU problems constrained by large-scale nonlinear PDEs under high-dimensional uncertainties.

In this work, we investigate the feasibility of overcoming the above computational challenges via a new development of neural operators. Neural operators have gained significant interest in recent years because of their remarkable ability to efficiently approximate high-dimensional mappings such as those arising in parametric PDE problems. 
We propose a novel framework that utilizes neural operators in the solution of PDE-constrained OUU problems. 
In this framework, the neural operator learns a mapping over a product measure of the random parameter distribution and an auxiliary distribution for the optimization variables that covers the feasibility sets of the optimization problems for which the neural operator is constructed. 
This presents a challenging task, since the mapping is over a product measure that is formally infinite-dimensional in at least one of the inputs. 
Mesh-dependent neural network strategies often suffer from deteriorating performance as the dimension of the problem increases. 
To avoid this difficulty, we use reduced bases to encode or compress the high-dimensional uncertain parameter field and the PDE state, which has gained traction as a popular architectural strategy for neural operators \cite{ BhattacharyaHosseiniKovachkiEtAl2021,FrescaManzoni2022,HesthavenUbbiali18,LuXuhuiShengzeEtAl2022,OLeary-RoseberryDuChaudhuriEtAl22,OLeary-RoseberryVillaChenEtAl22}. 
Moreover, a key contribution is to train the neural operator not only on the input-output solution map but also on its derivative with respect to the optimization variables. 
Building on recent work, the derivative training data can be efficiently computed and imposed in the training process by using reduced basis neural operators \cite{OLearyRoseberryChenVillaEtAl22}. 
This derivative-informed neural operator (DINO) strategy allows us to obtain significantly improved approximations not only of the joint parametric map but also of derivative quantities such as gradients, which are essential for efficient optimization methods in high dimensions.

We demonstrate the efficiency and accuracy of our proposed method on three challenging OUU problems subject to random field uncertainties; an optimal source control of a semilinear elliptic PDE, and an optimal boundary control of flow around a bluff body governed by Navier--Stokes equations in both two and three space dimensions.
We consider a challenging risk measure, the conditional value at risk (CVaR), which typically requires a large number of samples for accurate optimal solutions. 
We show that our neural operators offer reductions in execution time by factors of $10^{3}-10^{7}$ depending on the specific PDEs. 
Moreover, for the same quality of the OUU solution, the neural operators are over $10 \times$ more cost-efficient than a traditional PDE-based optimization method over a single optimization run, even after factoring in the construction cost for the neural operator.
Derivative training proves to be critical for accurate approximation of the solution operator and more importantly its derivatives, 
yielding OUU solutions of much higher quality than their counterparts without the derivative training.
Once trained, the neural operators can be reused to solve a family of OUU problems, e.g. with alternative optimization objectives and risk measures, at virtually no additional cost.

\subsection{Related work}

In recent years there has been significant work on developing neural operators for approximating high-dimensional, complex parametric maps arising in PDE problems \cite{FrescaManzoni2022,JinMengLu2022,KovachkiLiLiuEtAl2021,LiKovachkiAzizzadenesheliEtAl2020b,LiKovachkiAzizzadenesheliEtAl2020a,NelsenStuart2020,OLeary-RoseberryDuChaudhuriEtAl22,OLeary-RoseberryVillaChenEtAl22,RaissiPerdikarisKarniadakis2019,YuLuMengEtAl2022}. 
Additionally there has been interest in deploying neural operators to solve ``outer-loop'' problems such as Bayesian inverse problems \cite{CaoOLearyRoseberryJhaEtAl2022,LiKovachkiAzizzadenesheliEtAl2020a}, Bayesian optimal experimental design \cite{WuOLearyRoseberryChenEtAl2023}, 
optimal design \cite{DuOLearyRoseberryChaudhuriEtAl2023,OLeary-RoseberryDuChaudhuriEtAl22}.
In particular, neural operators have been considered as surrogates for a variety of deterministic PDE-constrained optimization problems in \cite{HwangLeeShinEtAl2022,KeilKleikampLorentzenEtAl2022,LuPestourieJohnsonEtAl22,ShuklaOommenPeyvanEtAl23,ZhaoLindellWetzstein22}, in addition to two recent works \cite{EigelHaaseNeumann22,GuthSchillingsWeissmann21} considering the use of neural networks 
for PDE-constrained OUU problems.


Specifically, in \cite{EigelHaaseNeumann22} the authors consider the use of neural networks for topology optimization under uncertain material parameters and loading conditions governed by linear elasticity using a phase-field based formulation. 
The neural networks output gradients with respect to the optimization variable given the states and gradients at previous iterations. 
This replaces the gradient computation step, which involves solving the adjoint PDE and an additional gradient PDE that arises as part of the phase-field formulation, thereby accelerating the optimization process.
We note that the approach of \cite{EigelHaaseNeumann22} differs from our proposed approach in that the state PDE is still solved during the optimization, as the states are used as inputs to the neural network.

Neural networks have also been considered for PDE-constrained optimization problems using an
``all-at-once'' or ``one-shot'' approach. 
Here, neural networks are used to represent the state and optimization variables, and potentially additional model parameters. 
The training of the neural network attempts to optimize the weights of the neural network using a loss function defined in terms of the optimization objective and the PDE residual, thereby trying to simultaneously achieve optimality (objective minimization) and feasibility (satisfying the PDE). 
Examples of this approach include \cite{HaoYingSuEtAl2022,LuPestourieYaoEtAl21,WangBhouriPerdikarisEtAl2021} for deterministic optimization problems and \cite{GuthSchillingsWeissmann21} in the OUU setting. 
This approach aims to solve single instances of the optimization problem, 
and the trained neural networks are not intended to be reused.
Instead, a new neural network needs to be trained for different choices of optimization performance objective, control cost, and constraints. 
Moreover, the training of the neural network may still be computationally expensive due to the ill-conditioning of the optimization problem involving both the optimization objective and the PDE residual.

The remainder of this paper is organized as follows. In Section \ref{section:ouu}, we introduce the formulation for PDE-constrained OUU problems. This is followed by a presentation of the proposed neural operator architecture in Section \ref{section:neural_operators}. Numerical results are then shown in Section \ref{section:numerical_results} before concluding with some remarks in Section \ref{section:conclusions}.

\section{PDE-constrained optimization under uncertainty}\label{section:ouu}

We consider systems that consist of a state $u \in \cU$, random parameters $m \in \cM$, and optimization/control variables $z \in \cZ$, where $\cU, \cM, \cZ$ are the respective Hilbert spaces to which they belong. 
The system is governed by a PDE 
written abstractly as 
\begin{equation}\label{eq:pde}
  R(u, m, z) = 0,
\end{equation}
where $R : \cU \times \cM \times \cZ \rightarrow \cU'$ is a differential operator and $\cU'$ is the dual of $\cU$. 
We assume that 
the PDE \eqref{eq:pde} admits a unique solution 
$u = u(m,z): \cM \times \cZ \rightarrow \cU$ for each given $m$ and $z$, i.e. the PDE problem is well-posed 
in the spaces of $m$ and $z$.

The uncertainty in $m$ is described by its distribution $\nu_{m}$, which is a measure over the Borel sigma algebra $\cB(\cM)$. 
In this work, we consider $m$ to be a random field, in particular a Gaussian random field, $\nu_m = \cN(\bar{m}, \cC)$, with a mean $\bar{m} \in \cM$ and a covariance operator of Mat\'ern class $\cC = \cA^{-\alpha}$ \cite{LindgrenRueLindstroem11}, where $\cA$ is an elliptic differential operator, e.g., $\cA = -\gamma \Delta + \delta I$ with Laplacian $\Delta$, identity $I$, and homogeneous Neumann boundary condition. The parameters $\alpha, \gamma, \delta > 0$ control the smoothness, variance, and correlation of the random field. We consider a performance function $Q:\cU\rightarrow \bR$ as a function of the state variable $u(m, z)$ that measures the performance of the system at a given $m$ and $z$. 
Due to the stochasticity of $m$, the quantity $Q(u(m,z))$ is a random variable. 
Thus, the PDE-constrained optimization problem is typically formulated in terms of a risk or statistical measure $\rho(Q)$ of 
$Q$. 
We can therefore write the PDE-constrained OUU problem as
\begin{equation}\label{eq:ouu}
  \min_{z \in \cZ_{ad}} \cJ(z) := \rho(Q)(z) + \cP(z),
\end{equation}
where $\cZ_{ad} \subset \cZ$ is an admissible set of the optimization variable $z$ and $\cP(z)$ is a penalization or regularization term that controls the cost or regularity of $z$. 



\subsection{Risk measures}


In the OUU problem, the risk measure $\rho$ quantifies the uncertainty in the performance function $Q$ due to the random parameter $m$. 
This effectively specifies a statistical goal we have in optimizing the PDE system over the distribution $\nu_m$. For example, one may be interested in optimizing the average performance of the system. This can be achieved using the expectation
\begin{equation}
  \rho_{\text{Mean}}(Q)(z)= \bE_{\nu_m}[Q(u(\cdot, z))]
\end{equation}
as a risk-neutral measure. 
In engineering applications, large values of $Q$ often correspond to undesirable or even failure states of the system. 
Although occurrence of such events may be rare, they can have catastrophic consequences. 
In such cases, it is insufficient to consider the expectation alone. Instead one can use risk-averse measures that account for the risk of large deviations from the mean.


In this work, we will consider the superquantile, or conditional value-at-risk (CVaR) \cite{RockafellarUryasev00}. Originally developed for financial risk management, the CVaR has become of a risk measure of interest in engineering applications and PDE-constrained OUU \cite{ChaudhuriKramerNortonEtAl22, KodakkalKeithKhristenkoEtAl22, KouriSurowiec16, LeeKramer23}. 
For a value $\beta \in [0,1]$, the $\beta$-quantile of $Q$, or the $\beta$-value at risk, is defined as
\begin{equation}\label{eq:VaR}
  \mathrm{VaR}_{\beta}[Q](z) := F^{-1}_{Q(u(\cdot, z))}(\beta),
\end{equation}
where $F_{Q((\cdot, z))}$ is the cumulative distribution function of $Q$.
The CVaR is then defined as the conditional expectation of $Q$ given that it exceeds the $\beta$-quantile, i.e.,
\begin{equation}\label{eq:CVaR}
  \rho_{\text{CVaR},\beta}(Q)(z) := \mathrm{CVaR}_{\beta}[Q] = \frac{1}{1-\beta} \bE_{\nu_m}[Q(u(\cdot, z)) 1_{Q(u(\cdot, z)) > \mathrm{VaR}_{\beta}}[Q](z)].
\end{equation} 
The value of $\beta$ specifies the level of risk-aversion. For $\beta = 0$, the superquantile is simply the expectation with the weakest risk aversion (risk neutral), while for $\beta = 1$, it is the essential supremum with the strongest risk aversion (worst case scenario). 

The OUU problem \eqref{eq:ouu} with the CVaR risk measure $\rho_{\text{CVaR},\beta}(Q)$ in \eqref{eq:CVaR} can be equivalently formulated as \cite{RockafellarUryasev00}
\begin{equation}\label{eq:CVaR_ouu}
  \min_{\substack{z \in \cZ_{ad} \\ t \in \bR}} \cJ_{\text{CVaR},\beta}(z,t) := t + \frac{1}{1-\beta} \bE_{\nu_m}[(Q(u(\cdot,z)) - t)^{+}] + \cP(z)
\end{equation}
by introducing an additional optimization variable $t \in \bR$, where $(\cdot)^{+} = \max(\cdot, 0)$. The cost functional \eqref{eq:CVaR_ouu} is non-differentiable because of the maximum function and smooth approximations of the maximum function are often used instead, e.g., \cite{KouriSurowiec16}
\begin{equation}
(x)^{+}_{\epsilon} = 
\begin{cases}
0 & \text{if } x < 0, \\
\left({x^3}/{\epsilon^2} - {x^4}/{2\epsilon^3} \right) & \text{if } 0 < x < \epsilon, \\
x - \epsilon/2 & \text{if } x \geq \epsilon, \\
\end{cases}
\end{equation}
with $\epsilon \ll 1$. 
This approximation has continuous second order derivatives, making the optimization amenable to gradient-based optimizers.
For simplicity, we will assume that $\epsilon$ is a fixed value (e.g. $\epsilon = 10^{-4})$ such that the CVaR approximation is sufficiently accurate. In practice, one may need to solve the OUU problem with successively decreasing values of $\epsilon$ to obtain a more accurate solution.

In this work, we focus on the CVaR risk measure, but note that our framework applies to a wide class of risk measures that can be formulated in terms of expectations of functions of $Q$. This includes moments of $Q$, probability of failure, buffered probability of failure \cite{RockafellarRoyset10}, and so we refer to \cite{ChenRoyset21, ShapiroDentchevaRuszczynski21} for a more extensive exposition of risk measures for optimization under uncertainty.

\subsection{Sample average approximation}
We compute the risk measures via sample average approximation (SAA), i.e. approximating the risk measure $\rho$ by a Monte Carlo estimator $\widehat{\rho}$. For concreteness, in the case of the mean, the estimator is  simply
\begin{equation}\label{eq:saa}
  \widehat{\rho}_{\text{Mean}}(z) := \frac{1}{N}\sum_{i=1}^{N} Q(u(m_i, z)),
\end{equation}
where $m_i \sim \nu_m$ are i.i.d. samples. The SAA optimization problem then becomes
\begin{equation}\label{eq:ouu_saa}
  \min_{z \in \cZ_{ad}} \widehat{\cJ}(z) := \widehat{\rho}_{\text{Mean}}(z) + \cP(z).
\end{equation}
Similarly, for the CVaR risk measure, we use the sample average approximation for $\cJ_{\text{CVaR}, \beta}$ along with the smoothed maximum function to obtain
\begin{equation}\label{eq:CVaR_saa_ouu}
  \min_{\substack{z \in \cZ_{ad} \\ t \in \bR}} \widehat{\cJ}_{\text{CVaR},\beta}(z) := t + \frac{1}{N} \sum_{i=1}^{N} \frac{1}{1-\beta}(Q(u(m_i,z)) - t)_{\epsilon}^{+} + \cP(z).
\end{equation}
Note that when the samples $\{ m_i \}_{i=1}^{N}$ are fixed during the solution of the optimization problem, \eqref{eq:ouu_saa} and \eqref{eq:CVaR_saa_ouu} become deterministic optimization problems, and can be solved using conventional algorithms for deterministic PDE-constrained optimization. 

\subsection{Gradient-based optimization for OUU}

It is well known that for high-dimensional optimization problems, derivative-free methods suffer from very slow convergence, 
requiring numerous function evaluations to obtain a solution. 
On the other hand, gradient-based methods such as Newton and quasi-Newton methods can typically attain asymptotic superlinear convergence rates that are independent of the dimension of the discretized optimization variable. 
Since each optimization iteration involves solving the state PDE for each of the $N$ samples $\{m_i\}_{i=1}^{N}$, it is imperative that the number of optimization iterations is kept small, necessitating the use of derivative-based optimization methods when the $z$ is high-dimensional.

Under the assumption that the risk measures can be formulated as expectations of functions of $Q$, i.e., $\rho(Q) = \mathbb{E}_{\nu_m}[f(Q)]$, and that they are differentiable with respect to $z$, the gradients can be computed by the chain rule
\begin{equation}\label{eq:ouu_risk_gradient}
  D_z \rho(Q) = D_z \mathbb{E}_{\nu_m}[f(Q)] = \mathbb{E}_{\nu_m}\left[D_{Q} f \; \partial_z Q\right],
\end{equation}
where the expectation is replaced by a sum in the case of SAA.
For the computation of $\partial_z Q$, we can write this (interpreted formally in the infinite dimensional setting) as 
\begin{align}\label{eq:ouu_qoi_gradient}
  g^T_z(m,z) &:= \partial_z Q(u(m,z)) = \partial_u Q(u(m,z)) \partial_z u(m,z) \nonumber \\
  &= -\partial_u Q(u(m,z)) \left[\partial_u R(u,m,z)\right]^{-1} \partial_z R(u,m,z) \nonumber \\
  &= p^T \partial_z R(u,m,z),
\end{align}
where we have used $p$ to denote the adjoint variable, which is given by solving the adjoint system
$
  p = -\left[\partial_u R(u,m,z)\right]^{-T}\partial_u Q(u(m,z))^T.
$
Thus, for a given $m$, $z$, and $u(m,z)$, the dominant cost of computing the gradient $g_z(m,z)$ is in solving an additional linear PDE involving adjoint operator $\partial_u R(u,m,z)^{T}$. The gradients can then be used in gradient-based optimization methods to minimize the SAA of the cost functional.



In summary, to estimate the risk measure via SAA, one needs to solve the state PDE for each sample $m_i$, amounting to $N$ state PDE solves. Moreover, computing the gradient of the risk measure with respect to the optimization variable requires an additional $N$ adjoint PDE solves, one for each sample $m_i$. 
Assuming a fixed sample size $N$, the overall cost of the OUU problem then involves $N_{\text{opt}} \times N$ state and $N_{\text{opt}} \times N$  adjoint PDE solves, where $N_{\text{opt}}$ is the number of optimization iterations, assuming a quasi-Newton method and ignoring the solves needed in the line search.




Solving OUU problems thus becomes computationally prohibitive for large $N$ and $N_{\text{opt}}$ when each PDE solve is expensive, which motivates the approximation of the map from the product space of the random parameter field and optimization variables to the PDE state, $(m,z) \mapsto u(m,z)$, by surrogates that are accurate for not only the PDE state but also its derivative with respect to the optimization variables $z$. 
This leads to the development of derivative-informed neural operators in the following section.

\section{Derivative-informed neural operators}\label{section:neural_operators}
Neural operators are neural network approximations of operators as mappings between input and output function spaces.
Let $x \in \mathcal{X}$ denote an input function in the function space $\cX$ 
equipped with a probability measure $\nu_x$, and $y \in \mathcal{Y}$ denote an output function in the function space $\cY$. 
For a given operator mapping $T:\mathcal{X} \rightarrow \mathcal{Y}$, the goal of neural operator learning is to construct an approximation $T_w$ parametrized by neural network parameters or weights $w$ that is optimal in a parametric Bochner space, e.g. $L^2_{\nu_x} = L^2(\mathcal{X},\nu_x ; \mathcal{Y})$.
That is, one seeks a solution to the expected risk minimization problem,
\begin{equation}\label{eq:no_exp_risk_min}
	\min_w \|T_w - T\|^2_{L^2_{\nu_x}} := \int_\mathcal{X}\|T_w(x) - T(x)\|_{\cY}^2 \, d\nu_x(x). 
\end{equation}
Due to the intractability of directly integrating with respect to the measure $\nu_x$, one typically approximates the 
risk minimization problem \eqref{eq:no_exp_risk_min} using finitely many samples of input-output pairs, $\{(x_i, T(x_i))\}_{i=1}^{N_s}$, leading to empirical risk minimization. 
We refer to this as the data-driven approach. 
Alternatively, equivalent objective functions such as norms of the PDE residual or other physics-based objective functions can be used as loss functions in the so-called ``physics-informed'' machine learning approach \cite{RaissiPerdikarisKarniadakis2019,YuLuMengEtAl2022}.
Hybrid approaches combining the data-driven approach with additional physics-informed losses have also been considered \cite{LiZhengKovachkiEtAl2021}. 
However, for simplicity and without loss of generality, we consider only the data-driven approach.

Neural operator construction typically consists of the following challenges:
(1) For solution operators of PDEs, generating input-output pairs requires solving the PDE, which is typically computationally expensive. Thus, one may be able to afford only a limited number of samples, leading to sampling errors in the empirical risk minimization problem. 
(2) The neural network training problem, i.e., the empirical risk minimization problem, is non-convex, and is solved using a nonlinear stochastic optimization method. Global optimization is typically NP-hard and one can only settle for local minimizers. Moreover, the neural operator training problem can be very sensitive to the optimization procedure, initial guesses, and additional factors such as scaling of the data.
(3) Appropriate neural operator architectures are required in order to accommodate a sufficiently rich functional representation.

In recent years many neural operators have been developed that construct effective representations of parametric PDE maps by encoding a priori known mathematical structure of the maps into the architecture. For maps that admit compressible representations in linear bases of the inputs and outputs $\mathcal{X}$ and $\mathcal{Y}$, reduced basis neural operators have been developed \cite{ BhattacharyaHosseiniKovachkiEtAl2021,FrescaManzoni2022, HesthavenUbbiali18, LuXuhuiShengzeEtAl2022,OLeary-RoseberryDuChaudhuriEtAl22,OLeary-RoseberryVillaChenEtAl22}. Other architectures have exploited compact nonlinear representations, such as a Fourier representation \cite{LiKovachkiAzizzadenesheliEtAl2020a}. 
These various neural operators often come with universal approximation theorems \cite{BhattacharyaHosseiniKovachkiEtAl2021,KovachkiLiLiuEtAl2021,LuJinKarniadakis2019,OLeary-RoseberryDuChaudhuriEtAl22}, asserting that mappings in a parametric Bochner space (e.g. $L^2_{\nu_x}$) can be approximated arbitrarily well by finite dimensional neural network representations. Finding suitable representations in practice remains a challenge, and is highly problem-dependent.

\subsection{Use of neural operators for OUU}\label{sec:neural_operator_ouu}
In this work, we consider the use of neural operators to approximate the solution map $u_w(m,z) \approx u(m,z)$ in the OUU problem (i.e. $\cY = \cU$).
Specifically, we require the neural operator to be sufficiently accurate over the input spaces $\cX = \cM \times \cZ$ equipped with a joint probability measure $\nu_x = \nu_m \otimes \nu_z$,
where we have introduced an auxiliary distribution $\nu_z$ from which training data for the control variables $z$ are generated. 
We note that although in certain contexts, it may be sufficient to learn directly the scalar performance function $Q(u(m,z))$, learning the full solution map is a more general approach. Neural operators can be constructed independent of the optimization objective/performance function, allowing for flexibility in its deployment and amortization of construction costs to solve families of optimization problems involving different choices of performance functions.

The selection of $\nu_z$ is problem dependent, but should contain in its support the admissible set $\cZ_{ad}$.
For example, when $z$ has bound constraints, a natural choice for $\nu_z$ is the uniform distribution over the bounds.
The choice of $\nu_z$ can additionally encapsulate prior information about regions of $\cZ$ over which we need the neural operator to be most accurate.
For simplicity, we will use uniform or Gaussian distributions supported over the admissible set,
and defer a study on strategically selecting $\nu_z$ to future work. 

The neural operator then replaces the PDE solution in the sample average approximation of the risk measures. In the case of the expectation, we take 
\begin{equation}\label{eq:neural_operator_risk}
	\bE[Q](z) \approx \frac{1}{N} \sum_{i=1}^{N} Q(u_w(m_i,z)),
\end{equation}
where $m_i \sim \nu_m$. 
Analogous to \eqref{eq:ouu_qoi_gradient}, the gradient of the neural operator based approximation for the performance function can be computed by chain rule, i.e.,
\begin{equation}\label{eq:neural_operator_gradient}
	\partial_z Q(u_w(m,z)) = \partial_u Q(u_w(m,z)) \; \partial_z u_w(m,z),
\end{equation}
where the derivative $\partial_z u_w(m,z)$ can be efficiently computed by automatic differentiation. 
The evaluation of the neural operator $u_w(m,z)$ and its derivative $\partial_z u_w(m,z)$ can typically be orders of magnitdue faster than solving the corresponding PDEs, allowing us to use a large sample size to solve the OUU problem with SAA and gradient-based optimization methods.
This can effectively eliminate sampling error in SAA, albeit at the cost of introducing a bias due to the approximation error of the neural operator.
Thus, the effectiveness of using the neural operator to solve OUU problems depends on the trade-off between this sampling error and the approximation error, which will be investigated in detail in the numerical experiments in Section \ref{section:numerical_results}.

Additionally, the gradient of the performance function $\partial_z Q$ with respect to the optimization variable $z$ plays an important role in the OUU problem. In particular, the optimal solution $z^*$ is characterized by the first-order necessary condition
$ D_z \cJ(z^*)= 0$. 
Here, the gradient of the cost functional $\cJ$ involves the gradient of the performance function, $\partial_z Q(u(m,z))$, as illustrated in \eqref{eq:ouu_risk_gradient}. 
Poor approximation of the gradient by the neural operator $\partial_z Q(u_w(m,z))$ leads to spurious local minimizers of $\cJ(z)$, resulting in inaccurate optimal solutions. 
The gradient error due to the neural operator approximation can be bounded in terms of the approximation error of the neural operator and its derivative as follows.
\begin{proposition}\label{prop:gradient_error}
	Assume that the operator $u(m,z)$ is differentiable and the performance 
	$Q : \cU \rightarrow \bR$ is Lipschitz continuously differentiable with constant $L_Q^1$. Then, at a given $(m,z)$, 
	\begin{equation}\label{eq:gradient_error}
		\|\partial_z Q (u) -\partial_z Q(u_w)\|_{\mathcal{Z}'}
		\leq L_Q^{1} \left\| \partial_z u \right\|_{\cL(\cZ, \cU)}  \|u - u_w\|_{\cU} + 
		 \|\partial_u Q(u_w)\|_{\cU'} \| \partial_z u - \partial_z u_w \|_{\cL(\cZ, \cU)}
	\end{equation}
	where $\| \cdot \|_{\cL(\cZ, \cU)}$ denotes the operator norm of bounded linear operators from $\cZ$ to $\cU$.
	\end{proposition}
	\begin{proof}
	The bound follows from a triangle inequality,
	\begin{align}
		\|\partial_z Q(u) &-\partial_z Q(u_w)\|_{\mathcal{Z}'}
		= \left\|\partial_u Q(u) \partial_z u - \partial_u Q(u_w) \partial_z u_w\right\|_{\cZ'} \nonumber \\
		&\leq \left\|(\partial_u Q(u)  - \partial_u Q(u_w)) \partial_z u\right\|_{\cZ'}
		+ \left\|\partial_u Q(u_w) (\partial_z u - \partial_z u_w)\right\|_{\cZ'} \nonumber \\
		&\leq L_Q^{1} \left\|\partial_z u\right\|_{\cL(\cZ, \cU)} \|u  - u_w\|_{\cU}
		+ \|\partial_u Q(u_w)\|_{\cU'} \left\|\partial_z u - \partial_z u_w\right\|_{\cL(\cZ, \cU)} \nonumber
	\end{align}
\end{proof}
This suggests that accuracy of both the solution operator $u(m,z)$ and its Jacobian $\partial_z u(m,z)$ are needed in order to guarantee accuracy of the gradient $\partial_z Q$. 

In light of the aforementioned points, we focus on the following interrelated challenges in deploying neural operators for the task of solving OUU problems subject to high-dimensional uncertainty. First, due to the large computational costs of forward simulations, as well as the high-dimensionality of the random parameter field $m$ and state $u$, one is faced with the task of learning complex high-dimensional operators from limited samples. 
Second, as the neural operator is to be deployed in the solution of an optimization problem, we are concerned not just with the operator approximation accuracy, but also the derivatives of the operator with respect to the optimization variable $z$.
It remains to establish whether or not sufficiently accurate neural operators can be constructed for OUU in a cost effective manner when compared to solving the OUU problem directly with the PDE.

\subsection{Derivative-informed neural operators}
In order to accurately and efficiently solve OUU problems it is important that neural operator errors do not lead to inaccurate approximations of the risk measure and its gradient. 
To this end, we propose to train the neural network approximation on not only evaluations of the solution operator, but also its derivative with respect to the optimization variable $z$. 
Until recently, parametric derivative training was not addressed in neural operator learning. The work of \cite{OLearyRoseberryChenVillaEtAl22} investigates the feasibility of constructing derivative-informed neural operators (DINOs), where neural operators are trained on both the operator value and its derivative with respect to the input variable.
That is, the operator regression task is performed in $H^1_{\nu_x}$ instead of $L^2_{\nu_x}$ as in \eqref{eq:no_exp_risk_min}, i.e.,
\begin{equation} \label{eq:generic_dino_exp_risk_min}
	\min_w 
		\|T_w - T\|^2_{H^1_{\nu_x}} := \int_\mathcal{X}\|T_w(x) - T(x)\|_{\cY}^2 + \underbrace{\|D_x T_w(x) - D_x T(x)\|_{\text{HS}(\mathcal{X},\mathcal{Y})}^2}_\text{Jacobian error term} d\nu_x(x)
	,
\end{equation}
where $\|A\|_{\HS(\cX, \cY)}$ denotes the Hilbert--Schmidt norm of a linear operator $A \in \cL(\cX, \cY)$ defined using an orthonormal basis $e_i$ of $\cX$ as
$\|A\|_{\HS(\cX, \cY)}^2 := \sum_i \langle A e_i, A e_i \rangle_{\cY}^2.$

In the OUU setting, we consider the following derivative-informed neural operator training 
motivated by the error expression in \eqref{eq:gradient_error}, 
\begin{align} \label{eq:ouu_dino_exp_risk_min}
	\min_w  
	\|u_w - u\|^2_{H^{(0,1)}_{\nu_m\otimes \nu_z}} := & \int_{\mathcal{M}\times \mathcal{Z}}\|u_w(m,z) - u(m,z)\|_{\cU}^2 \\
	 + &\underbrace{\|\partial_z u_w(m,z) - \partial_z u(m,z)\|_{\text{HS}(\mathcal{Z},\mathcal{U})}^2}_{\text{control Jacobian error term}} d\nu_m\otimes \nu_z(m,z) \nonumber
	 ,
\end{align}
where we refer to the derivative of the state with respect to the control variables, $\partial_z u$, as the control Jacobian. 
This objective is ostensibly intractable in comparison to the $L^p_{\nu_x}$ learning problem \eqref{eq:no_exp_risk_min}, due to the computational costs of evaluating $\partial_z u$ for training data, and the large online memory and arithmetic costs associated with computing and differentiating through the Jacobian error term. 
However, we show that the data generation and training costs may be significantly reduced using reduced basis architectures, which we introduce in Section \ref{sec:reduced_basis_architecture}.


Besides improving accuracy of the Jacobian, it is numerically shown in \cite{OLearyRoseberryChenVillaEtAl22} that the Jacobian training also improves the generalization accuracy of the neural operator output itself. 
Since the optimal solution $z^*$ is unlikely to coincide with any of the training samples, generalization accuracy of the neural operator within the admissible space $\cZ_{\ad}$ is important to obtaining accurate OUU solutions.
In this regard, the Jacobian $\partial_z u(m, z)$ contains additional information about the dependence of $u$ on $z$, which can be particularly valuable in preventing overfitting of the neural network when the training sample size is small. Often, this information can be obtained at little additional cost, as discussed in Section \ref{sec:efficient_jacobian_training}.


\subsection{Reduced basis architectures}\label{sec:reduced_basis_architecture}

In order to address the high dimensionality of the input and output spaces, 
we propose the use of reduced basis neural operator architectures that exploit the intrinsic low dimensionality of the map 
$m,z \mapsto u(m,z)$.
In this framework, neural networks are used to approximate the mapping between reduced basis representations of the input and output spaces.
Since the construction of the neural network depends only on the intrinsic dimensionality of the solution mapping, 
reduced basis architectures have the capability to learn complex high-dimensional PDE-based maps in an efficient and dimension independent manner when intrinsic dimensionality is low.

For the OUU problem,
we assume that the spaces $\cM$ and $\cU$ have dimensions $d_M$ and $d_U$, arising from the discretization of infinite dimensional function spaces.
In particular, $d_M$ and $d_U$ may be arbitrarily large as the mesh resolution increases. 
On the other hand, we consider the space $\cZ$ to be inherently finite dimensional, with dimension $d_Z \ll d_U, d_M$. 
This is typical of many optimization problems in engineering, since design/control choices are often finite-dimensional by nature. 
Therefore, we consider reduced basis representations of $\cM$ and $\cU$, but do not employ dimension reduction on $\cZ$. Despite this, our method extends to the case where $\cZ$ is infinite dimensional by applying similar dimension reduction techniques to $\cZ$. 

We can write the proposed reduced basis neural operator as 
\begin{equation}\label{eq:rbnet}
	u_w(m,z) = \Phi_{r_U}\varphi_w(m_r,z) +b,  \qquad m_r = \Psi_{r_M}^T m,
\end{equation}
where $\Psi_{r_M} \in \bR^{d_M \times r_M}$ and $\Phi_{r_U} \in \bR^{d_U \times r_U}$ are reduced bases for the parameter and state spaces $\cM$ and $\cU$, respectively.
The mapping $\varphi_w:\mathbb{R}^{r_M \times d_Z} \rightarrow \mathbb{R}^{r_U}$ is a reduced basis neural network parametrized by weights $w \in \mathbb{R}^{d_W}$, and $b \in \mathbb{R}^{r_U}$ is a bias term.
We will refer to this architecture as the multi-input reduced basis neural operator (MR-NO) and use the name MR-DINO when the architecture is trained with the derivative-informed loss \eqref{eq:ouu_dino_exp_risk_min}.
A schematic for MR-DINO is shown in Figure \ref{fig:control_dino}. 

\begin{figure}[!htb]
	\center
	\begin{tikzpicture}[scale = 0.8, transform shape, every node/.style={draw,outer sep=0pt,thick}]
	\node[bag] (Input1) at (0,0.5) [minimum width=1cm,minimum height=3cm] {Uncertain parameter \\ reduced basis \\Eigenvectors of $\cC$ \\ $\mathbb{R}^{d_M} \rightarrow \mathbb{R}^{r_M}$};
	\node[bag] (Input2) at (0,-2.5) [minimum width=2cm,minimum height=2cm] {Optimization\\ variable\\ $z \in \mathbb{R}^{d_Z}$};
	
	\node[bag] (NN) at (5,0) [minimum width=3cm,minimum height=1.5cm] {Neural Network:\\ $\mathbb{R}^{r_M}\times \mathbb{R}^{d_Z}\rightarrow \mathbb{R}^{r_U}$};
	\draw (Input1.north east) -- (NN.north west);
	\draw (Input1.south east) -- ($(NN.north west) - (0,1.0)$);
	\draw (Input2.north east) -- ($(NN.north west) - (0,1.0)$);
	\draw (Input2.south east) -- (NN.south west);
	
	\node[bag] (Output) at (10,0)[minimum width=1cm,minimum height=3cm] {State \\ reduced basis\\ $\mathbb{E}_{\nu_m\otimes \nu_z}[uu^T]$ \\ $\mathbb{R}^{r_U} \rightarrow \mathbb{R}^{d_U}$};
	\draw (NN.north east) -- (Output.north west) (NN.south east) -- (Output.south west);
	
	\node[bag] (Explanation) at (7,-2.75) [minimum width=3cm,minimum height=1.5cm] {Derivative-Informed Neural Operator\\$\min_w\mathbb{E}_{\nu_m \otimes \nu_z}[\|u_w - u\|_2^2 + \|\partial_z u_w - \partial_z u \|_{\mathrm{HS}}^2]$};
	\end{tikzpicture}
	\caption{Schematic for the MR-DINO in the solution of OUU problems.}
	\label{fig:control_dino}
	\vskip -0.5cm
\end{figure}
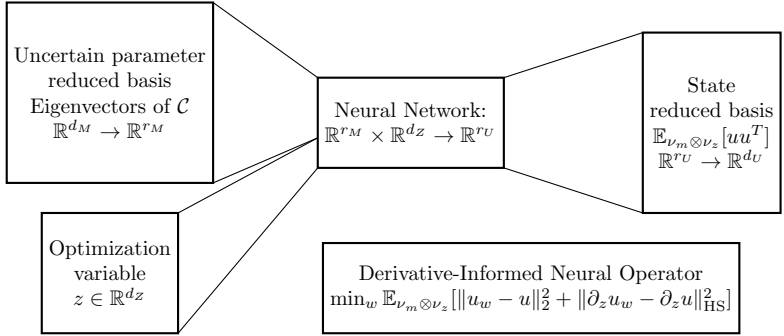

We note that for certain classes of performance functions, such as linear or quadratic forms, its evaluation at the neural operator output, i.e. $Q(u_w)$, can be further accelerated using the reduced basis representation of $u_w$. 
For example, in the quadratic case with $Q(u) = u^T W u$, $W \in \bR^{d_U \times d_U}$, and $b = 0$ for simplicity, we have
\begin{equation} Q(u_w) = u_w^T W u_w = \varphi_w^T \Phi_{r_U}^T W \Phi_{r_U} \varphi_w 
	= \varphi_w^T W_{r_U} \varphi_w, 
\end{equation}
where the reduced matrix $W_{r_U} := \Phi_{r_U}^{T} W \Phi_{r_U} \in \mathbb{R}^{r_U \times r_U}$ can be precomputed so that the costs of evaluating $Q(u_w)$ scale only with the rank $r_U$. 
This applies to many commonly used optimization objectives, such as $L^2(\Omega)$ norms and data misfits.

In this work, we obtain reduced bases by proper orthogonal decomposition (POD) for the state, and principal component analysis (PCA) for the uncertain parameter. This architecture is an extension of the PCANet \cite{BhattacharyaHosseiniKovachkiEtAl2021} to multiple inputs. 
The POD basis is computed by solving the eigenvalue problem
\begin{equation}
	\bE_{\nu_x}[(u - \bar{u})  (u - \bar{u})^T] \phi_i = \lambda_{u}^{(i)} \phi_i, 
\end{equation}
where $\bar{u} = \bE_{\nu_x}[u]$. 
The reduced basis is then taken as the $r_U$ eigenvectors $\Phi_{r_U} = [\phi_1, \dots \phi_{r_U}]$ corresponding to the $r_U$ largest eigenvalues, and the bias is taken as the mean $b = \bar{u}$.
In practice, the eigenvectors are computed by SVD of data matrix $U = [u_1, \dots u_N]$, where $\{u_i\}_{i=1}^{N}$ are snapshots of the solution from training data.
Moreover, when $u \in \mathbb{R}^{d_U}$ represents the coefficients of a finite element discretization, 
the inner product of $\cU$ becomes a weighted inner product $\langle u_1, u_2 \rangle_{M_u} := u_1^T M_u u_2$, 
where the symmetric positive definite matrix $M_u$ arises from the discretization of the underlying function space inner product. 
For $L^2(\Omega)$, $M_u$ is simply the mass matrix.
In this case, to maintain consistency with the infinite dimensional setting, we consider a weighted POD in which the SVD is carried out on $M_u^{1/2} U$, 
such that the resulting basis is orthonormal in the $\langle u_1, u_2 \rangle_{M_u}$ inner product.

Reduced basis networks based on POD have become very popular in neural operator learning 
\cite{BhattacharyaHosseiniKovachkiEtAl2021,FrescaManzoni2022,LuXuhuiShengzeEtAl2022,OLeary-RoseberryDuChaudhuriEtAl22,OLeary-RoseberryVillaChenEtAl22}. 
Approximation errors based on the POD truncation can be proven via the Hilbert--Schmidt Theorem or Fan's Theorem \cite{BhattacharyaHosseiniKovachkiEtAl2021,ManzoniNegriQuarteroni16,OLeary-RoseberryDuChaudhuriEtAl22}, with an upper bound by the sum of the trailing eigenvalues $\{\lambda_u^{(i)}\}_{i>r_M}$. More general a-priori reduced basis error analyses can be found in \cite{BinevCohenDahmenEtAl11} based on their comparison to Kolmogorov $n$-width for optimal linear approximations and for specific parametric problems with explicit exponential rates in \cite{ ChenQuarteroni14,MadayPateraTurinici02} and algebraic rates in \cite{ChenSchwab15, ChenSchwab16}, see review in \cite{ChenQuarteroniRozza17, ChenSchwab16c, CohenDeVore15, OhlbergerRave15}.

The basis for the random parameter field is analogously computed using the PCA. 
This uses the dominant $r_M$ eigenvectors of the covariance of $\nu_m$ as the reduced basis,
$\cC \psi_i = \lambda_{m}^{(i)} \psi_i,$
giving rise to the rank $r_M$ basis $\Psi_{r_M} = [\psi_1, \dots, \psi_{r_M}]$. 
In the case of a Gaussian random field, this simply corresponds to the Karhunen--Lo\`eve expansion (KLE) of the random field.
The KLE basis works well in cases where the most sensitive modes of the solution operator align well with the dominant modes of the covariance operator.
When this is not the case, low-dimensional bases that directly capture sensitivity of the solution operator $u$ with respect to the random parameter field $m$ can be identified using an active subspace approach \cite{OLeary-RoseberryVillaChenEtAl22}.
For simplicity, we will use only the KLE basis and note that our proposed approach works analogously with the active subspace basis, which may be an appealing option for problems where this $m$-sensitivity is informative.

\subsection{Efficient Jacobian training for reduced basis architectures} \label{sec:efficient_jacobian_training}
To efficiently perform Jacobian training, we use the fact that the range of the control Jacobian of the reduced basis neural operators, $\partial_z u_w$, is precisely the span of the output reduced basis. Therefore, for the reduced basis neural operator presented in \eqref{eq:rbnet}, the minimization of Jacobian error can be re-written in the reduced output space as
\begin{align}\label{eq:reduced_basis_derivative_equivalence}
	\min_w &
	\| u_w(m,z) - u(m,z) \|_{\cU}^2 + 
	\|\partial_z u_w(m,z) - \partial_z u(m,z)\|_{\text{HS}(\mathcal{Z},\mathcal{U})}^2 \\
	\Leftrightarrow & \min_w 
	\| \varphi_w(m_r,z) + \Phi_{r_U}^T M_u (b - u(m,z)) \|_{\ell^2}^2 \nonumber 
	+ \| \partial_z \varphi_w(m,z) - \Phi_{r_U}^T M_u \partial_z u(m,u) \|_F^2, \nonumber
\end{align}
due to Theorem 1 in \cite{OLearyRoseberryChenVillaEtAl22}, where $\| \cdot \|_F^2$ is the Frobenius norm. 
We refer to the term $\Phi_{r_U}^T M_u \partial_z u \in \mathbb{R}^{r_U \times d_Z}$ 
as the reduced control Jacobian of $u$.
Thus, the loss term in \eqref{eq:reduced_basis_derivative_equivalence} allows the online memory and arithmetic costs for training our reduced basis neural operators to scale only with $d_Z \times r_U$ instead of $d_Z \times d_U$. Morever, the reduced control Jacobian $\Phi_{r_U}^{T} M_u \partial_z u$ can be efficiently computed after computing the PDE solution $u(m,z)$ at relatively little additional cost.
Recall that by the implicit function theorem, we have 
\begin{equation} \label{eq:projected_sensitivity}
	\Phi_{r_U}^{T} M_u \partial_z u(m,z) = - \Phi_{r_U}^T M_u \left[\partial_u R(u,m,z) \right]^{-1} \partial_z R(u,m,z),
\end{equation}
where $\partial_z R$ is of size $d_U \times d_Z$. 
When $d_Z < r_U$, we solve the linearized state PDE, $\left[ \partial_u R(u,m,z) \right]^{-1} \partial_z R(u,m,z)$ for each of the $d_Z$ columns of $\partial_z R$, On the other hand, if $d_Z \geq r_U$, we instead solve the adjoint PDE, $[\partial_u R(u,m,z)]^{-T} M_u \Phi_{r_U}$ for each of the $r_U$ basis vectors $\Phi_{r_U}$. 
In any case, given the state $u(m,z)$, the additional computational costs of computing the Jacobian training data are associated with the $\min(d_Z, r_U)$ linearized PDE solves, all with the same linearized operator, $\partial_u R(u,m,z)$, or its adjoint.

For steady-state PDEs, which is the focus of this work, the costs to compute the Jacobian training data can be mitigated as follows. 
When the state PDE is linear, the linearized PDE has the same linear operator as the state PDE. 
Hence, when one uses a direct solver, the linearized PDE solves for the Jacobian computation can reuse the same triangular factors of the state PDE solution, which require only the back substitution step of the solve and become negligible in cost relative to the state PDE solve. 
For linear problems requiring iterative methods, the costs of the preconditioner construction can be amortized across the linearized PDE solves, since they use the same preconditioner as the linear state PDE solve.
When the state PDE is nonlinear, multiple linearized state PDEs need to be solved in order to arrive at a solution, e.g. via Picard or Newton iterations.
The linearized PDE solves involved in the Jacobian computation therefore cost only a fraction of the state PDE solve. We will demonstrate the cost of the Jacobian computation relative to that of the state PDE solve in the numerical experiments in Section \ref{section:numerical_results}.

\section{Numerical experiments}\label{section:numerical_results}

In this section, we demonstrate the accuracy and efficiency of our method using three PDE-constrained OUU problems; the optimal source control of a semilinear elliptic PDE in 2D with an uncertain diffusion coefficient field, and the optimal boundary control of fluid flow governed by Navier--Stokes equations in both 2D and 3D under an uncertain inflow velocity field. We present a detailed comparison of the accuracy and cost of our method against the ground truth solutions for the 2D cases which are still affordable, and demonstrate the power of our method to the 3D flow control problem where the ground truth solution is prohibitively expensive to compute.

For the numerical results, we use FEniCS \cite{LoggMardalGarth12} to implement the finite element discretizations with direct solvers from PETSc \cite{BalayAbhyankarAdamsEtAl23}.
Data generation is conducted with the use of hIPPYlib \cite{VillaPetraGhattas21} and hIPPYflow \cite{hippyflow}. Neural network approximations are implemented using TensorFlow \cite{tensorflow2015-whitepaper}. Unless otherwise specified, timings are carried out on a single compute node with an Intel Xeon Gold 6248R processor and NVIDIA A100 40GB GPU for the neural network computations.


\subsection{Cost-accuracy comparison of the neural operator}\label{sec:cost_accuracy_comparison}


We train neural operators using varying training data sizes, both with (MR-DINO) and without Jacobian loss (MR-NO). 
The neural operators are then used to solve the OUU problem by SAA of the cost functional using a large sample size. 
We denote optimal solutions computed by the neural operator by $z^*_{\NN}$. 
For comparison, we solve the OUU problems using the PDE by SAA with different sample sizes, representing PDE-constrained OUU under different computational budgets. 
We denote the optimal solutions computed using the PDE by $z^*_{\PDE}$. 

The accuracies of the approximations are evaluated with respect to the reference solutions, which are obtained by solving the PDE-based OUU problem with a large sample size for SAA. This reference optimal solution is denoted as $z^*_{\mathrm{ref}}$. The performance of the approximate optimal solutions $z_{\mathrm{approx}}^*$ are compared to $z^*_{\mathrm{ref}}$ using 
\begin{equation}
  \textrm{relative optimal cost error } := \frac{|\cJ(z_{\mathrm{approx}}^*) - \cJ(z^*_{\mathrm{ref}})|}{|J(z^*_{\mathrm{ref}})|}.
\end{equation}

We compare the computational cost in terms of the total number of state PDE solves used to obtain the optimal solution, which is the dominant cost of the optimization process. 
For the neural operator-based optimization, this corresponds to the total number of training samples. For the PDE-based optimization, 
it is given by the number of optimization iterations times the sample size for SAA.

Additionally in Section \ref{section:comparison_of_timings}, we present the computational costs of the state PDE solves in comparison to the costs of the linearized PDE solves required to compute Jacobian 
(vector products) used in both the PDE-based gradient computation and neural operator Jacobian training data generation using \cite{hippyflow}.
These costs are also compared to the neural operator training and evaluation costs. 


\subsection{Optimal source control of a semilinear elliptic PDE} \label{section:semilinear_elliptic}
We first consider the source control of a semilinear elliptic PDE over a square domain, $\Omega = (0, 1)^2$, with a log-normal parameter field. The PDE is given by 
\begin{equation}\label{eq:elliptic_pde}
  -\nabla \cdot ( e^{m} \nabla u) + r u^3 = \sum_{i=1}^{49} z_i f_i, 
\end{equation}
with homogeneous Dirichlet boundary conditions. The coefficient $r = 0.1$ is a constant reaction coefficient. The log permeability $m \sim \nu_{m} = \cN(\bar{m}, \cC)$ is a Gaussian random field for which we specify $\bar{m} = -1$ and $\cC = (-\gamma \Delta + \delta I)^{-2}$ with $\gamma = 0.1, \delta = 5.0$ for the distribution of $m$. 
The control variables $z = \{z_i\}_{i=1}^{49}$, for which we consider box bounds $\cZ_{\ad} = [-4, 4]^{49}$,
define the strengths of invidiual sources
$\{f_i\}_{i=1}^{49}$ in a $7 \times 7$ grid of localized Gaussian sources, with $f_i$ being given by 
\begin{equation}
  f_i(x) = \frac{1}{\sigma \sqrt{2\pi}}\exp\left( 
    -\frac{|x - x_i|^2}{2 \sigma^2}
  \right),
\end{equation}
where $x_i$ is the position of the source and $\sigma = 0.08$ is a width parameter.
The PDE is discretized in a $64 \times 64$ uniform triangular mesh using piecewise linear elements for both the random parameter and state spaces, leading to $d_U = d_M = 4225$. The performance function $Q$ is of the following tracking type,
\begin{equation}
  Q(u) := \int_{\Omega} (u - u_{\text{target}})^2 dx,
\end{equation} 
where $u_{\text{target}}$ is the target state for the system. 
This allows us to define an optimal control problem that minimizes the CVaR of the tracking objective, i.e.,
\begin{equation}
  \min_{z \in \cZ_{\ad}} \CVaR_{\alpha}[Q](z)
  \qquad \st \text{\eqref{eq:elliptic_pde}}.
\end{equation}
Note that we solve the CVaR optimization problem using the reformulation \eqref{eq:CVaR_ouu}. 
As examples, we consider a sinusoidal target state $u_{\mathrm{target}} = \sin(2 \pi x_1) \sin(2 \pi x_2)$, and a quadratic target state $u_{\mathrm{target}} = 4 x_2 (1 - x_2)$.


\subsubsection{Solution by neural operator}

We generate the training data for the neural operator by sampling from the input distribution $\nu_{m} \otimes \nu_{z}$, where we use the auxiliary distribution $\nu_{z} = \text{Uniform}(-4, 4)^{49}$ for the control variables. For the reduced basis generation, we use $512$ samples to compute the POD basis for the output and use the KLE basis for the input parameter, which are given as the discrete eigenfunctions of the covariance operator of $\nu_m$.

We train neural networks of two different architectural sizes. One is a small network with input rank 
$r_M = 50$ for the input basis and output rank $r_U = 100$, connected by a dense neural network with two hidden layers of width $200$ that approximates the mapping between the reduced basis coefficients. We also consider a larger network with $r_M = 100$ and $r_U = 300$, and two hidden layers of width $400$. All networks use softmax activation functions. 
We train neural networks using training data sets of size 512, 1,024, 2,048, and 4,096, both with and without Jacobian training. 
The neural networks are trained using Adam for 1,600 epochs with an initial learning rate of $10^{-3}$ that is reduced to $2.5 \times 10^{-4}$ after 800 epochs.
In Figure \ref{fig:poisson_training_errors}, we plot the mean relative $L^2(\Omega)$ errors of the neural operator and the relative errors of its Jacobian measured in the Hilbert--Schmidt norm. The errors are computed on a test set of 1,024 samples from the input distribution, and averaged over 10 different runs of the training process with different initializations. The incorporation of Jacobian training consistently leads to smaller errors in both the solution and its Jacobian. These improved approximations give rise to smaller bounds for the control gradient errors as discussed in Proposition \ref{prop:gradient_error}. 
Interestingly, we observe that without Jacobian training, the larger architecture performs worse than the smaller architecture as a consequence of overfitting. With Jacobian training, the small architecture saturates in accuracy, while the larger architecture yields lower errors and continues to improve. This suggests that the additional Jacobian training data allows one to use larger and more expressive architectures while mitigating the tendency of overfitting.


\begin{figure}
  \centering
  \begin{subfigure}{0.45\textwidth}
    \centering
    \includegraphics[width=\textwidth, trim=30 20 0 10]{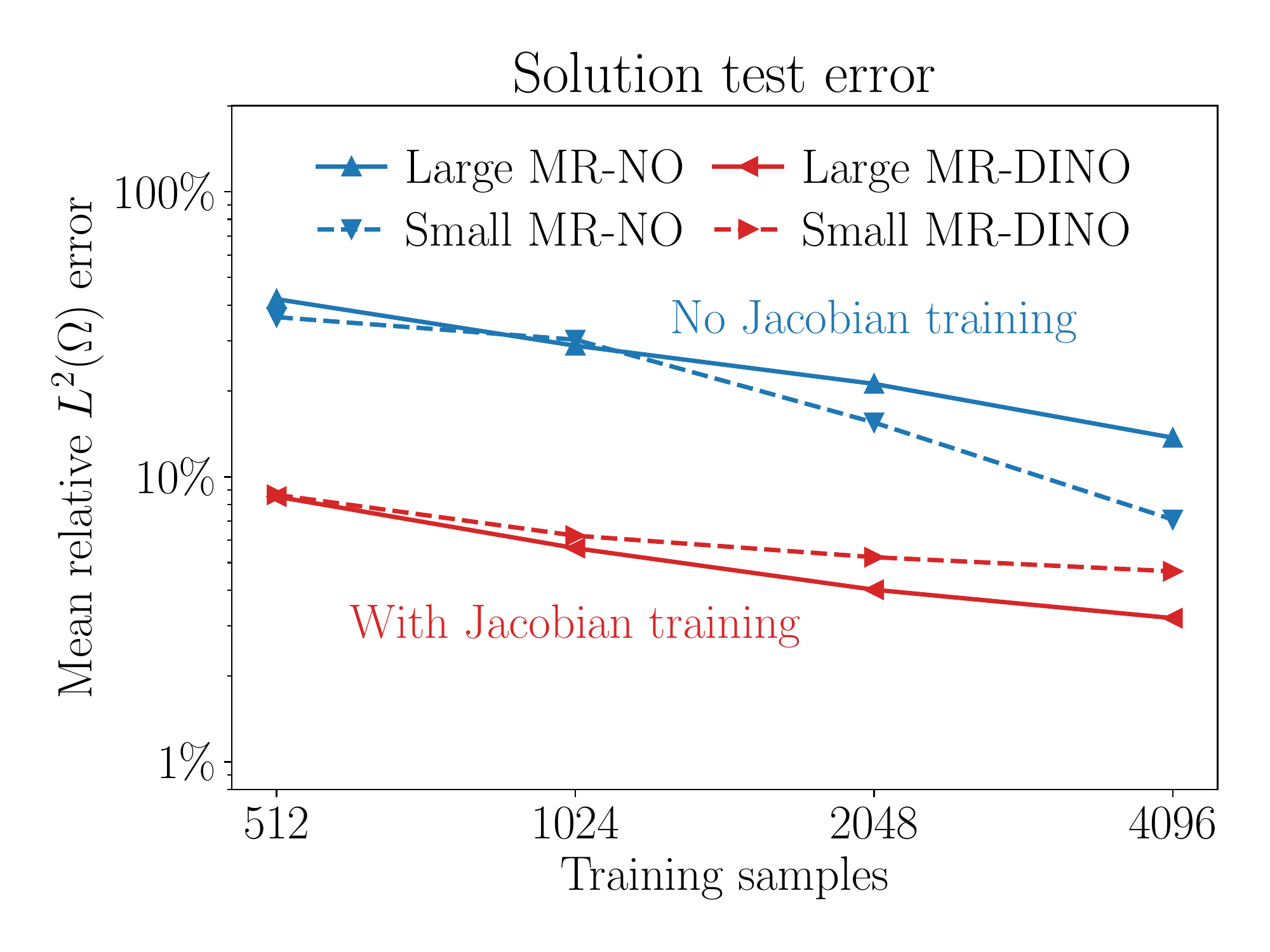}
  \end{subfigure}
  \begin{subfigure}{0.45\textwidth}
    \centering
    \includegraphics[width=\textwidth, trim=30 20 0 10]{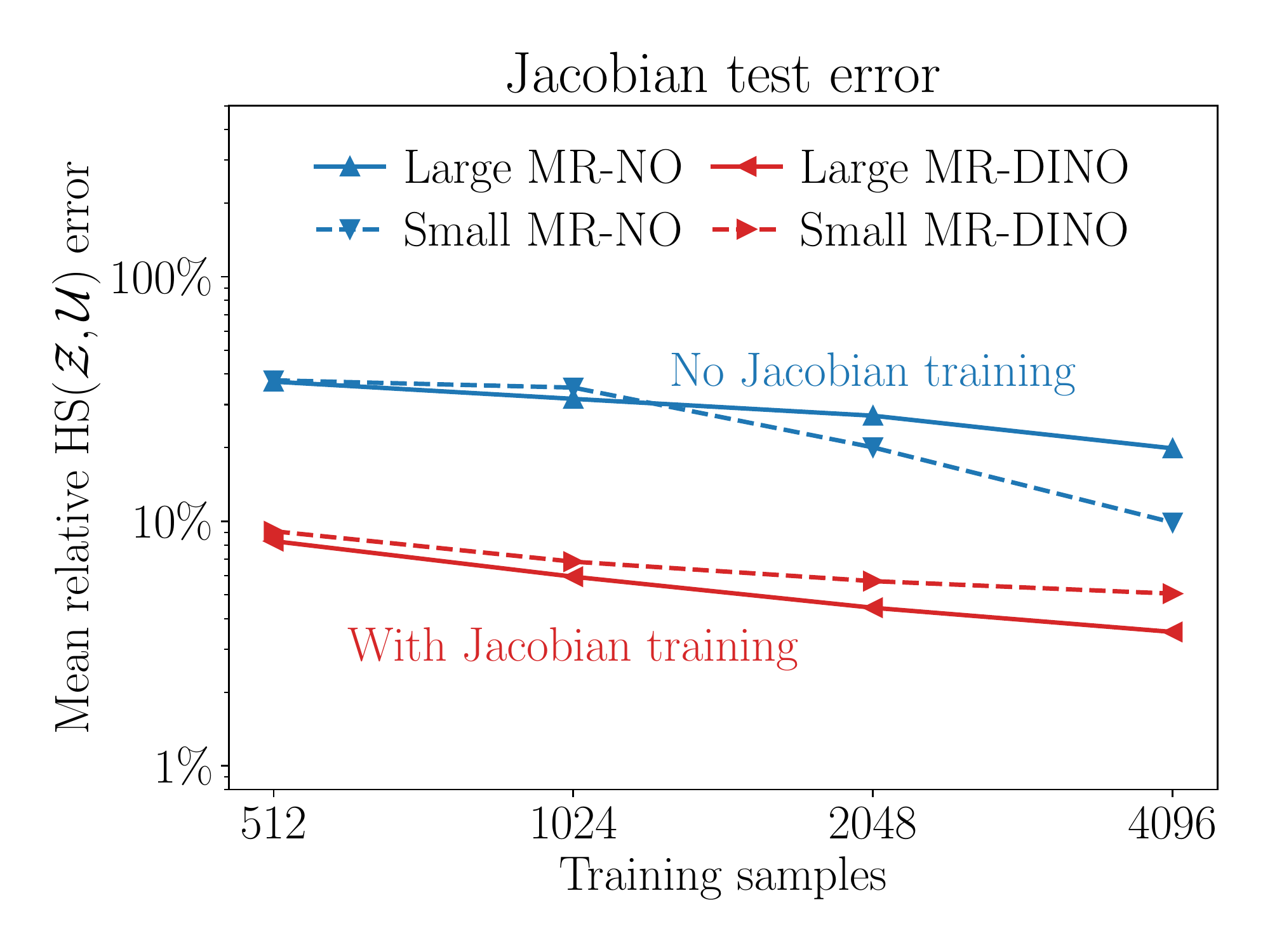}
  \end{subfigure}
  \caption{State ($L^2(\Omega)$) (left) and Jacobian ($\mathrm{HS}(\cZ, \cU)$) (right) testing errors of elliptic PDE neural operators trained with (MR-DINO) and without Jacobian training (MR-NO). 
  }
  \label{fig:poisson_training_errors}
  \vskip -0.5cm
\end{figure}


\begin{figure}[!htb]
  \centering 
  \begin{minipage}{0.3\textwidth}
    \begin{subfigure}{\linewidth}
      \includegraphics[width=\linewidth, trim=0 40 0 40]{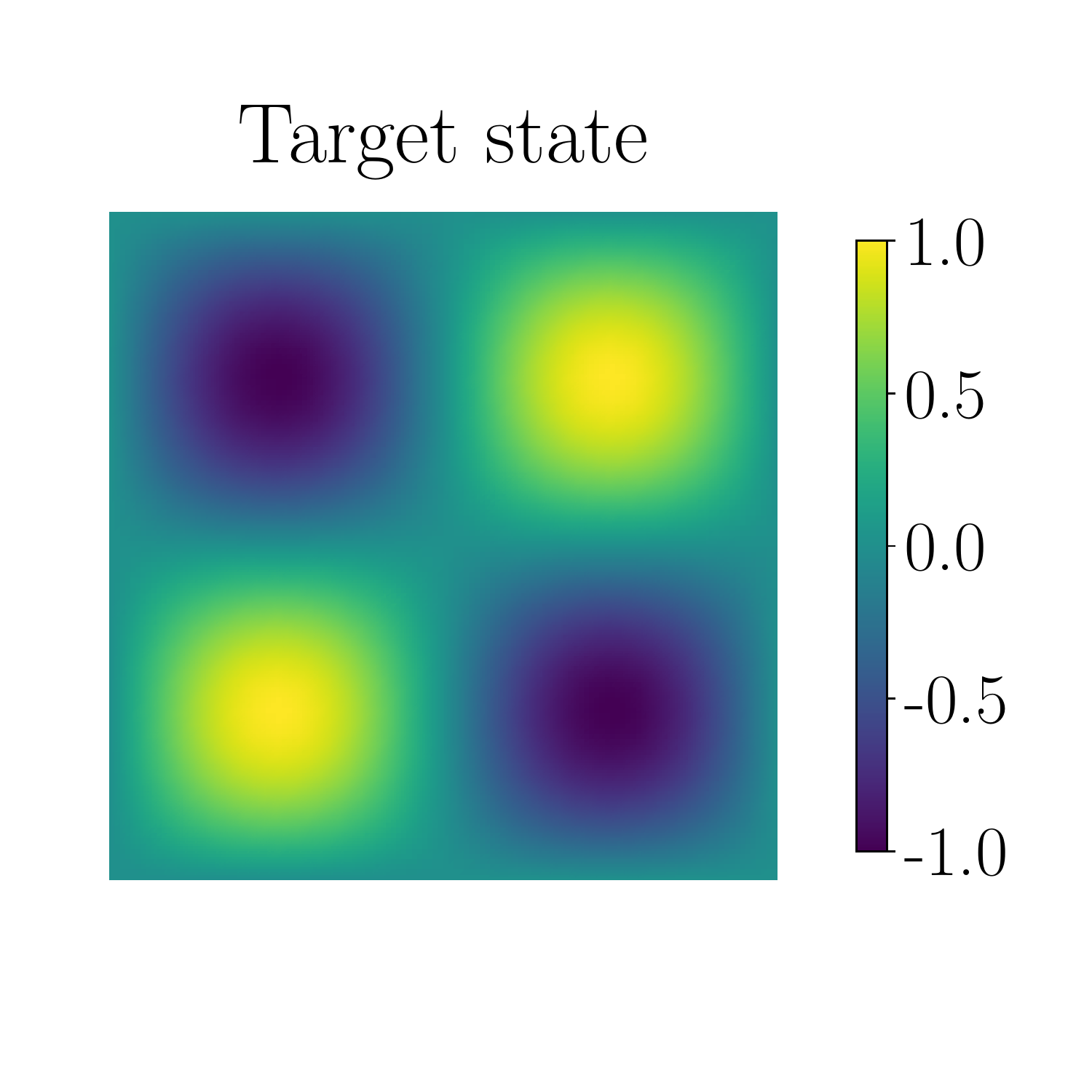}
    \end{subfigure}
    \begin{subfigure}{\linewidth}
      \includegraphics[width=\linewidth, trim=0 40 0 40]{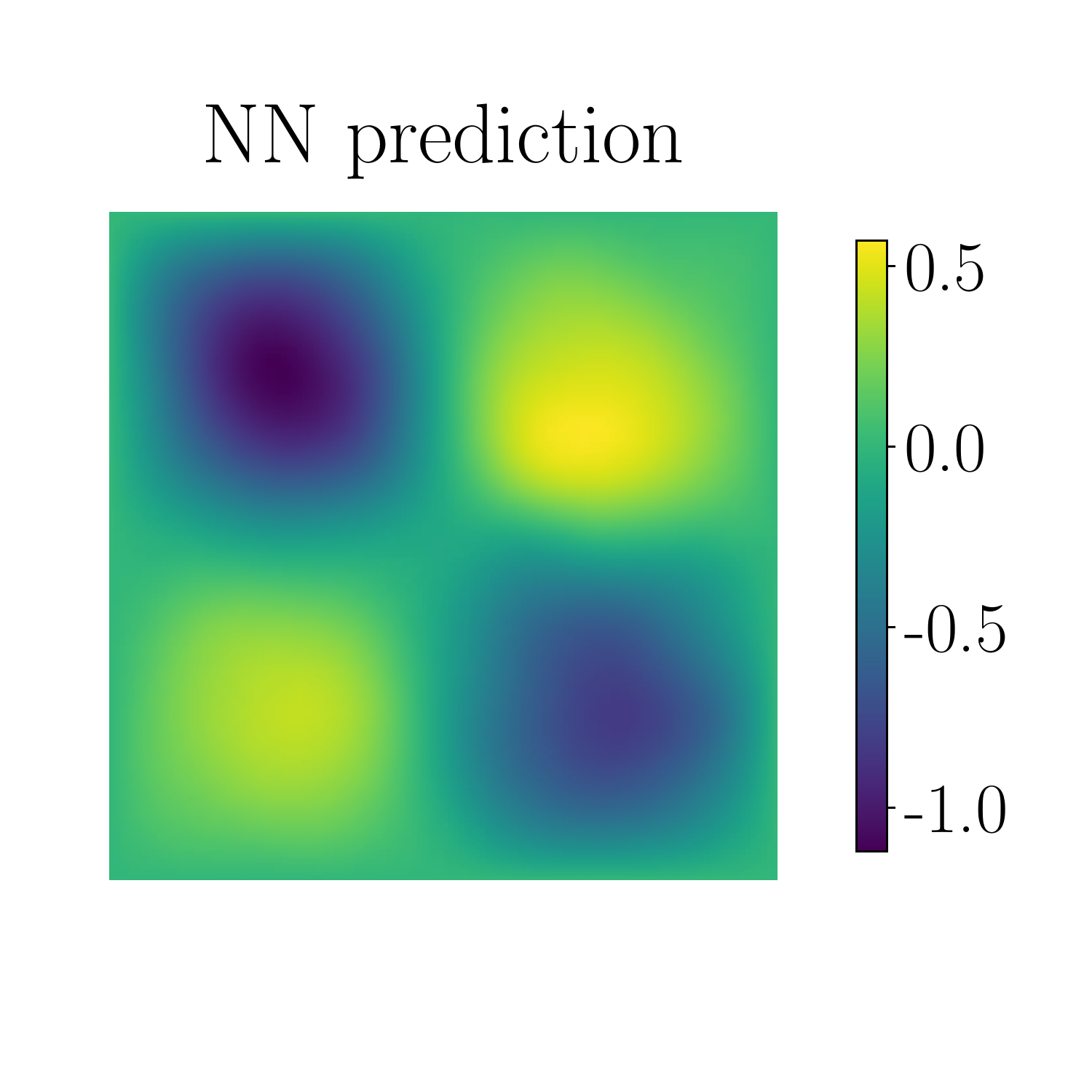}
    \end{subfigure}
  \end{minipage}
  \begin{minipage}{0.3\textwidth}
    \begin{subfigure}{\linewidth}
      \includegraphics[width=\linewidth, trim=0 40 0 40]{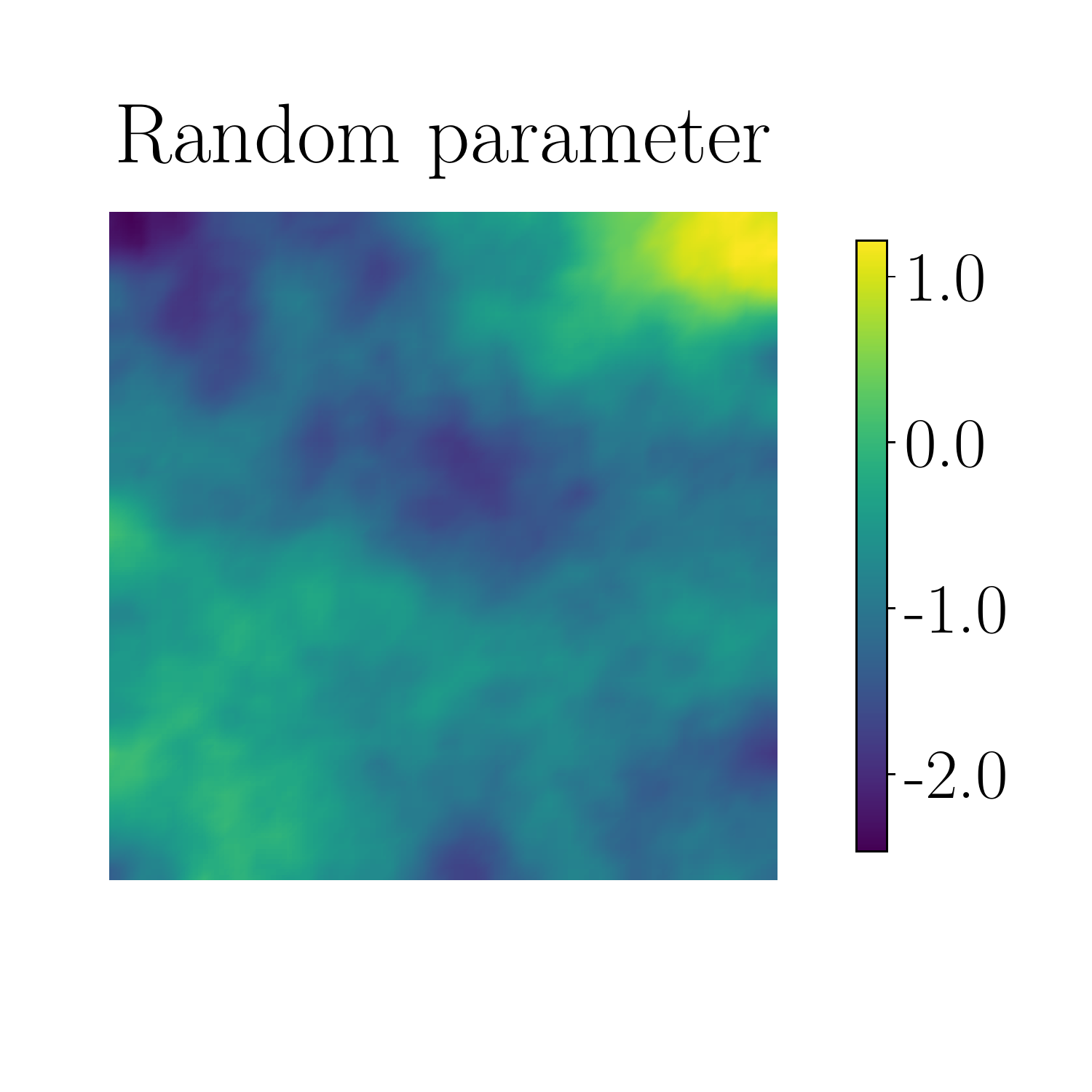}
    \end{subfigure}
    \begin{subfigure}{\linewidth}
      \includegraphics[width=\linewidth, trim=0 40 0 40]{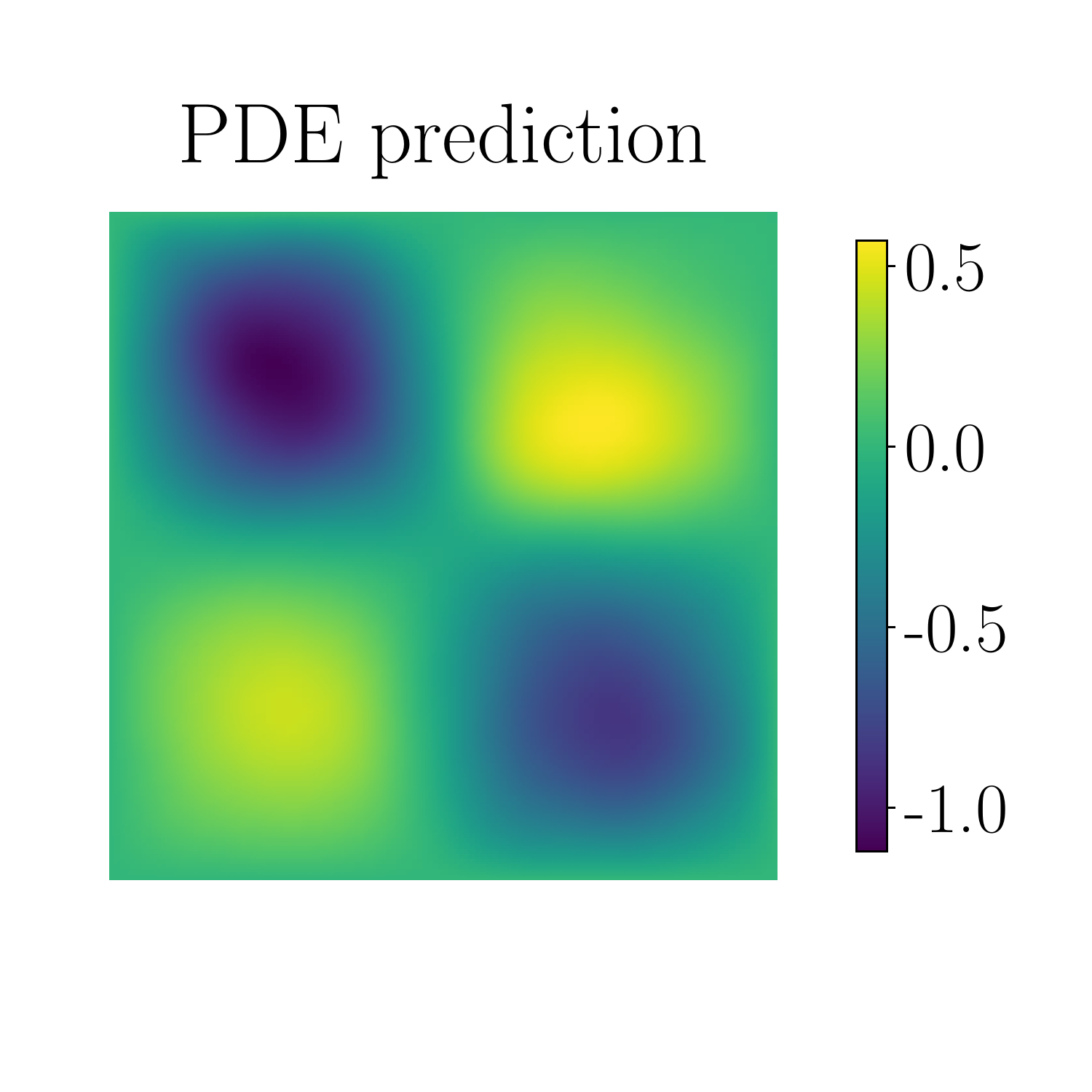}
    \end{subfigure}
  \end{minipage}
  \begin{minipage}{.35\textwidth}
    \includegraphics[width=\linewidth, trim=40 0 80 0]{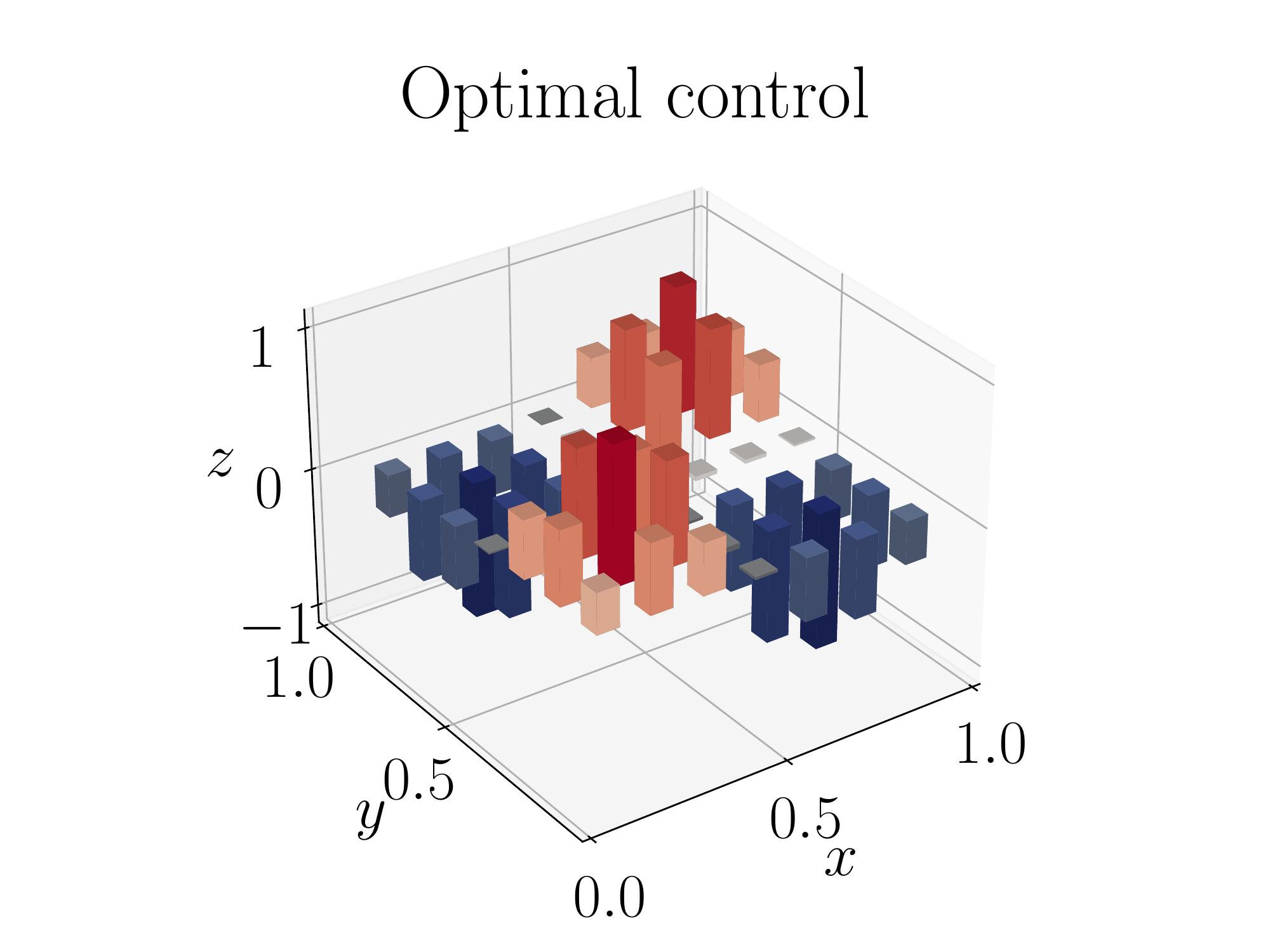}
  \end{minipage}
  \caption{Top-left: a sinusoidal target $u_{\text{target}}$. Top-middle: a random parameter sample $m$. Right: optimal control $z_{\NN}^{*}$ using the neural operator. Bottom-left and bottom-middle: neural surrogate $u_{w}(m,z_{\NN}^{*})$ and PDE solution $u(m,z_{\NN}^{*})$ at $m$ and $z_{\NN}^{*}$. The SAA optimization problem with MR-DINO is solved in 52 seconds.
  }
  \label{fig:ouu_poisson_sinusoid_states}
  \vskip -0.5cm
\end{figure}

The OUU problem is solved using the neural operator by SAA with a sample size of $N = 2,048$, using L-BFGS-B to obtain the optimal control $z_{\NN}^{*}$. For clarity of presentation, we only consider the larger network architecture with $r_M = 100$ and $r_U = 300$ for the remainder of this section. Figure \ref{fig:ouu_poisson_sinusoid_states} shows an example of the optimization solution for the sinusoidal target state with MR-DINO trained on 2,048 samples. 
In particular, we present the target state, the computed optimal control $z_{\NN}^{*}$, a sample of the random coefficient field $m$, and the state corresponding to the optimal control for the sampled $m$ as predicted by both the neural operator approximation and the true PDE. 
Comparing the neural network prediction of the state to the true PDE, we observe that visually, the neural operator can accurately approximate the PDE solution operator. 
Moreover, the computed minimizer aims to match the state to the target while being robust to the different possible realizations of $m$. 

\subsubsection{Cost-accuracy comparison with PDE solutions using SAA}
To quantify the accuracy of the neural operator OUU solution relative to its costs, we compare the CVaR 
at the optimal controls $z_{\NN}^{*}$ and
$z_{\PDE}^{*}$.
For the PDE, we solve the optimal control problem for $z_{\PDE}^{*}$ using SAA with sample sizes of $N = $16, 32, 64, 128 and 256.
An accurate reference solution of the OUU problem $z_{\mathrm{ref}}^{*}$ is also computed using SAA with $N = 4,096$.

\begin{figure}[!htb]
  \centering
  \begin{subfigure}{0.45\textwidth}
    \includegraphics[width=\linewidth, trim=30 20 0 10]{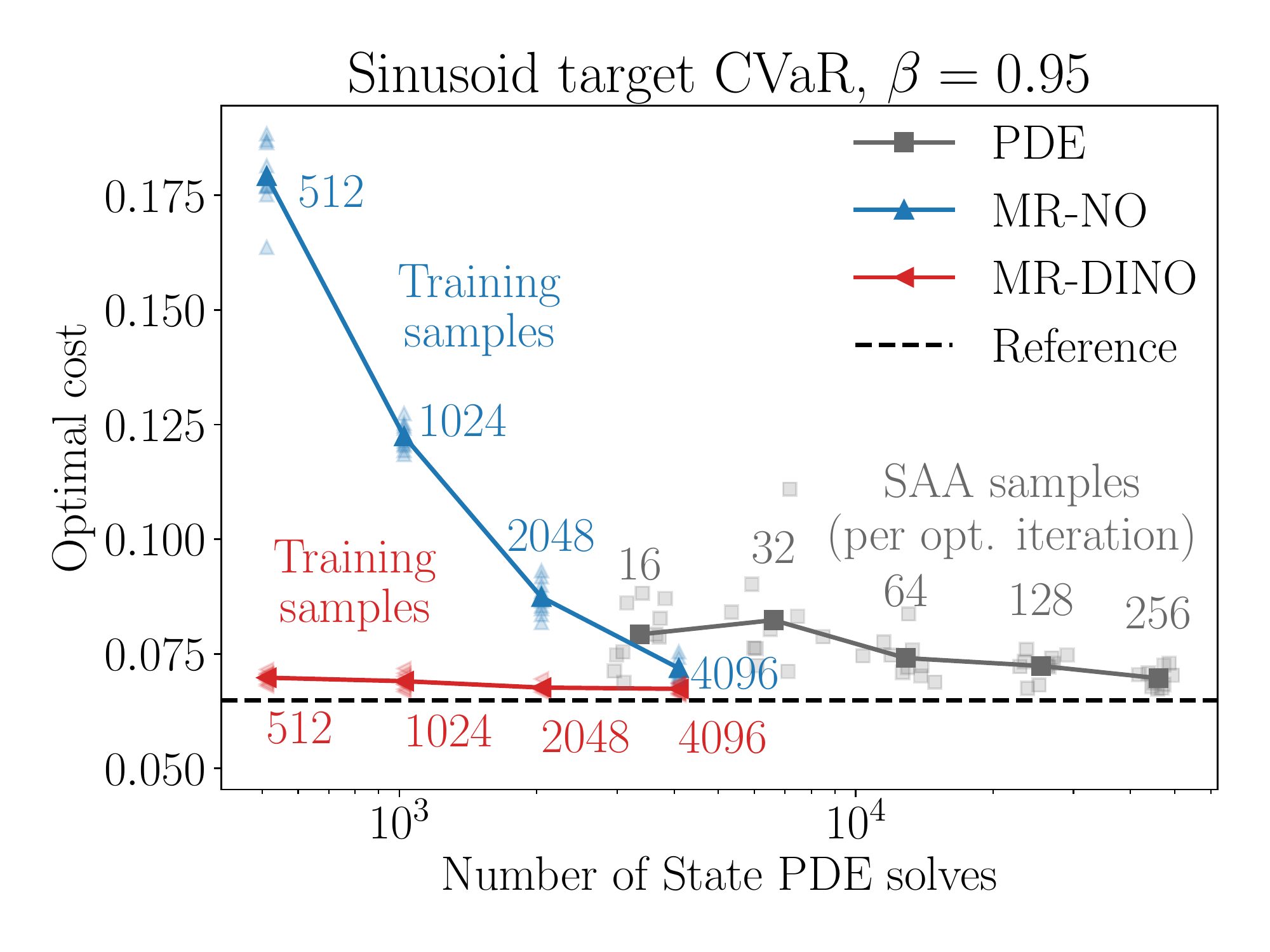}
  \end{subfigure}
  \begin{subfigure}{0.45\textwidth}
    \includegraphics[width=\linewidth, trim=30 20 10 10]{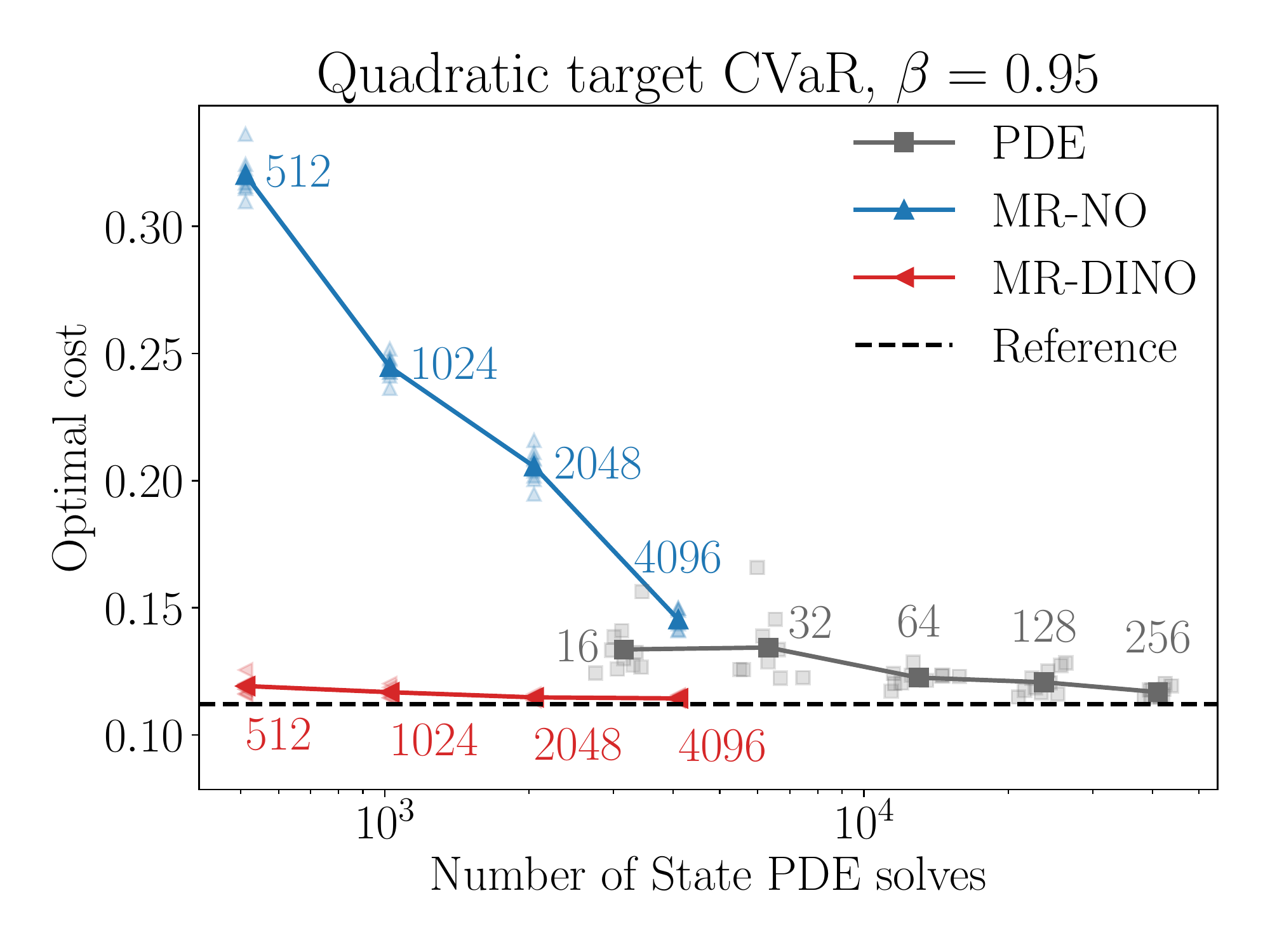}
  \end{subfigure}
  \caption{Optimal cost values for the sinusoidal (left) and quadratic (right) target states at optimal solutions versus the number of state PDE solves required. 
  }
  \label{fig:ouu_poisson_cvar}
  \vskip -0.5cm
\end{figure}

The CVaR values at optimal controls are plotted in Figure \ref{fig:ouu_poisson_cvar} as a function of the number of state solves required to obtain the optimal control. The CVaR values are evaluated by a Monte Carlo estimator with 8,192 samples, and are averaged across 10 different runs of the OUU problem with different sets of parameter samples. 
In the case of the neural operators, the training initializations are also randomized across the 10 runs. 
Notice that due to sampling errors, the CVaR values for the PDE-based optimal controls are suboptimal relative to the reference value, and tend to vary significantly between runs. 
The optimality improves as the sample sizes increase. 
On the other hand, the CVaR values from MR-DINO are much closer to the reference value, and are much more consistent across runs. This suggests that the bias committed by the neural operators is less problematic than the variance (sampling error) arising from the PDE. 
Moreover, for the same number of training samples, the MR-DINO with Jacobian training greatly improves the optimal solutions. 

\begin{figure}[!htb]
  \centering
  \begin{subfigure}{0.45\textwidth}
    \includegraphics[width=\linewidth, trim=20 20 10 10]{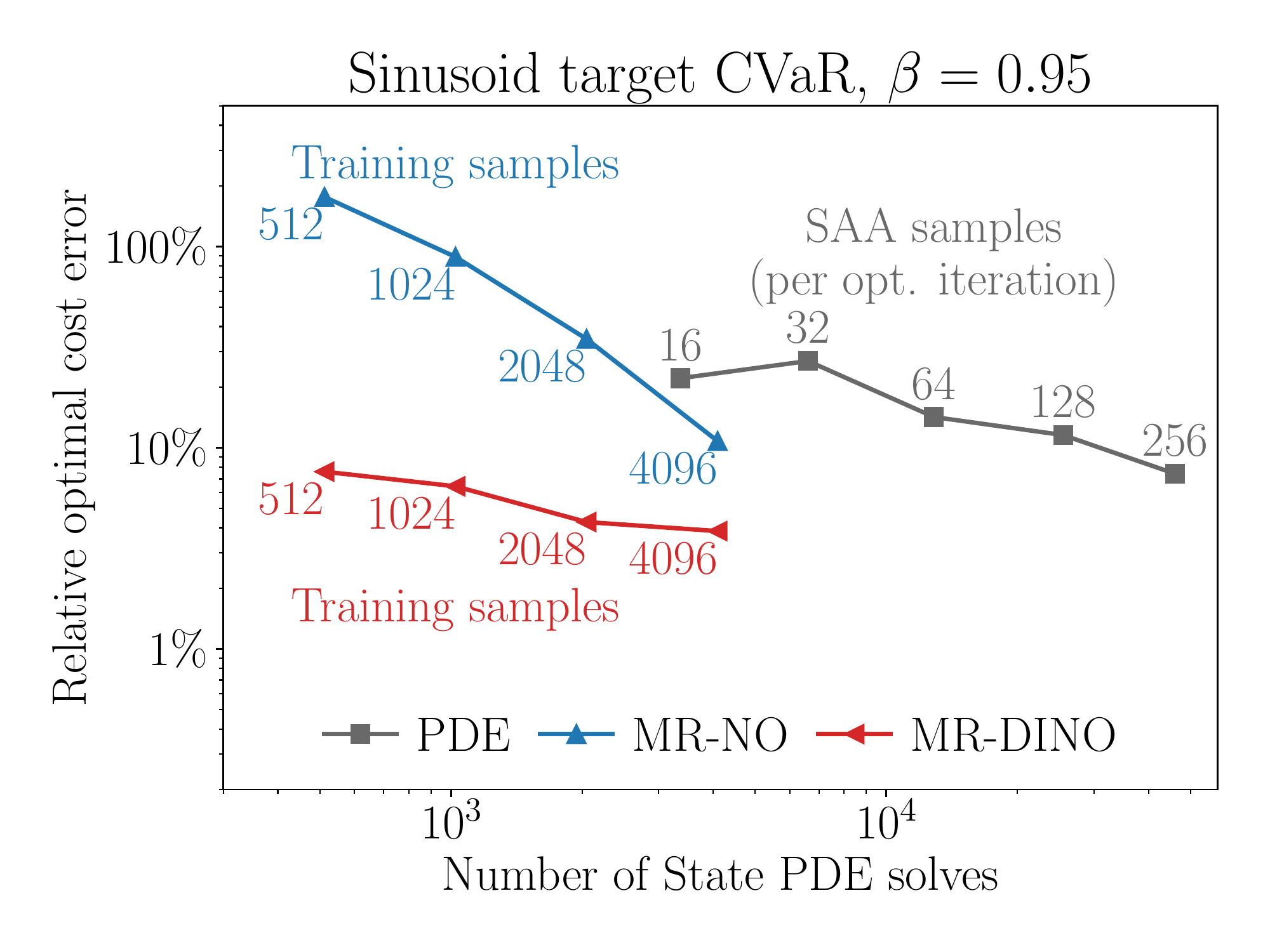}
  \end{subfigure}
  \begin{subfigure}{0.45\textwidth}
    \includegraphics[width=\linewidth, trim=20 20 10 10]{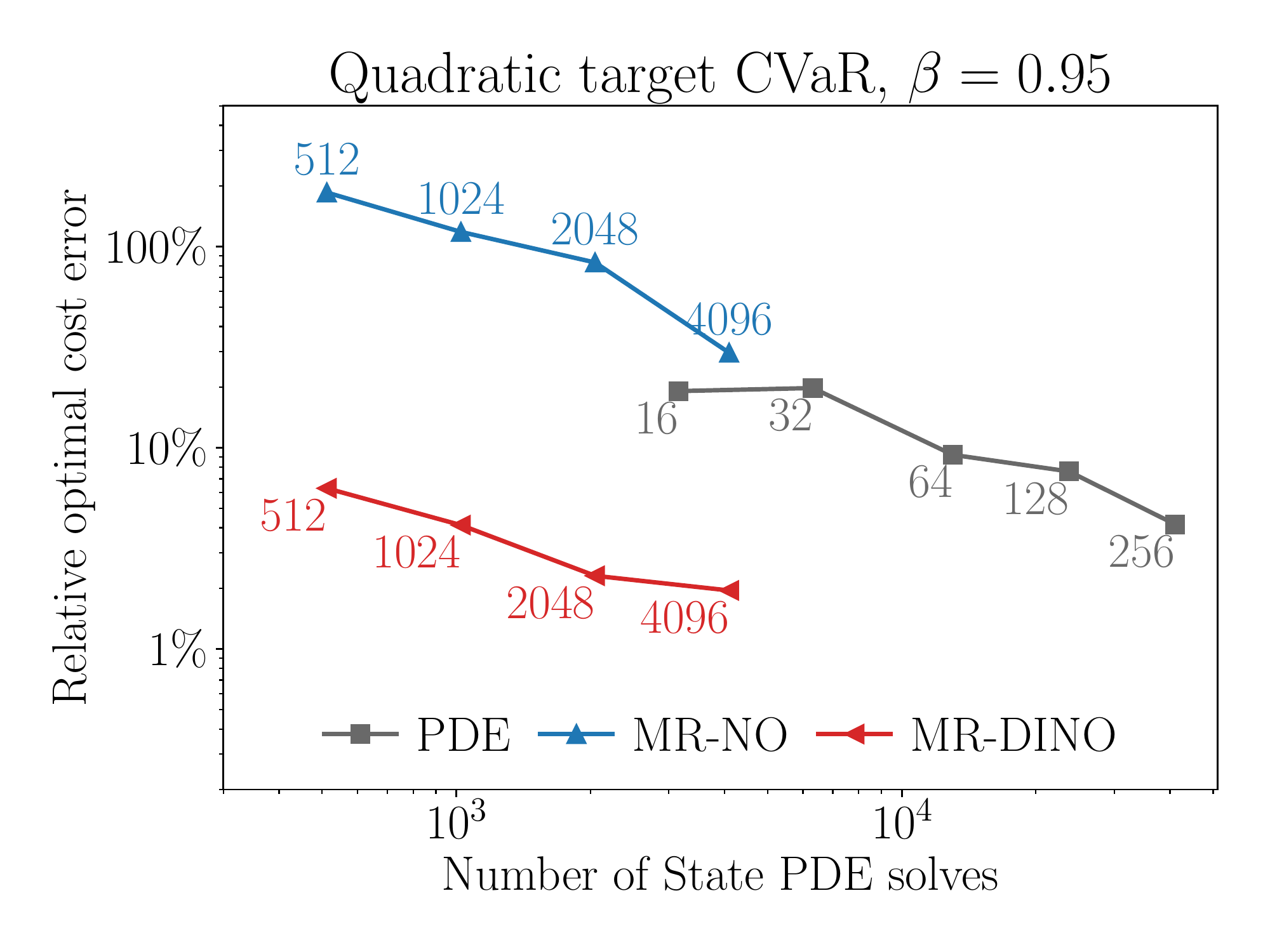}
  \end{subfigure}
  \begin{subfigure}{0.45\textwidth}
    \includegraphics[width=\linewidth, trim=20 20 10 10]{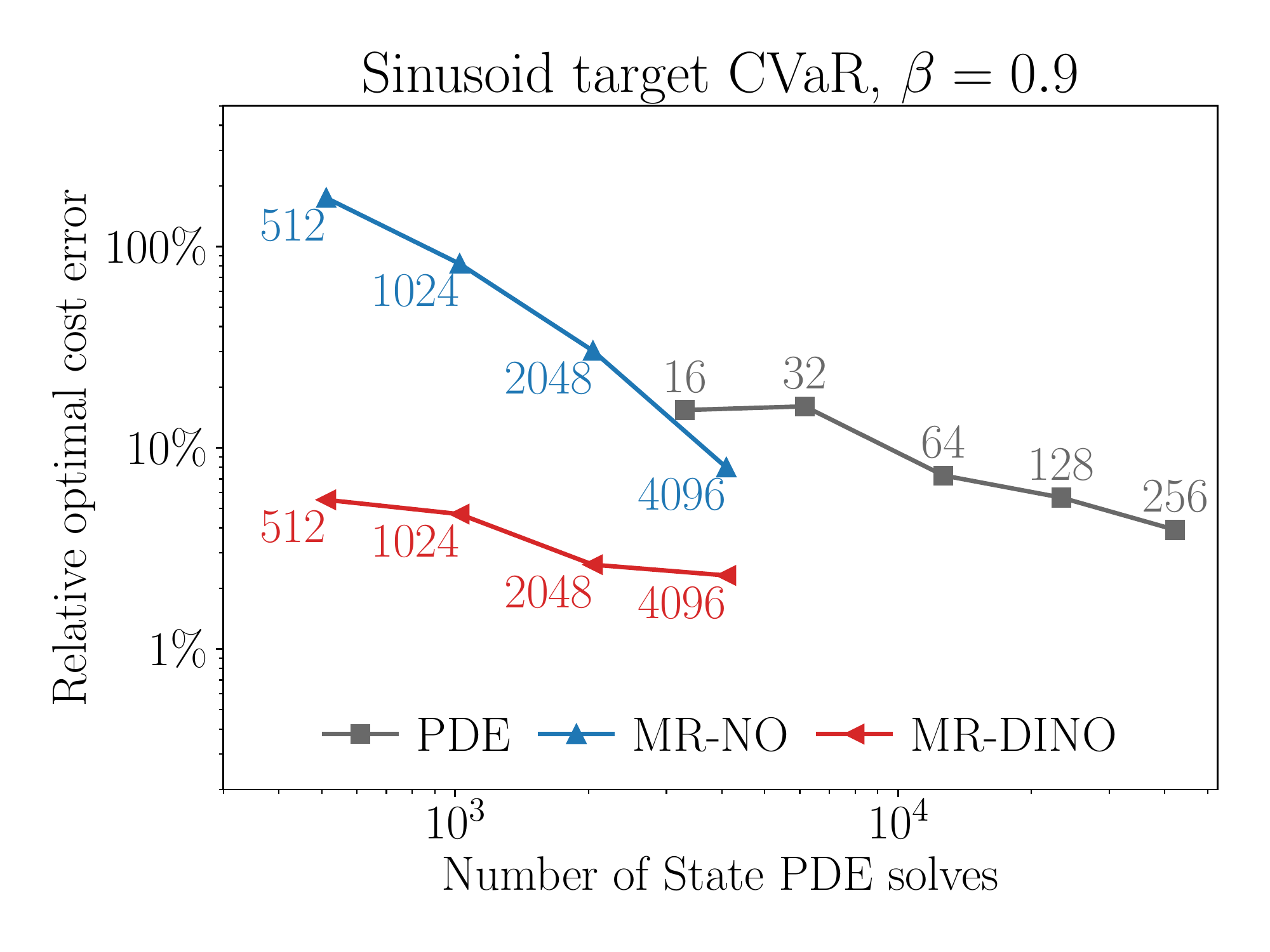}
  \end{subfigure}
  \begin{subfigure}{0.45\textwidth}
    \includegraphics[width=\linewidth, trim=20 20 10 10]{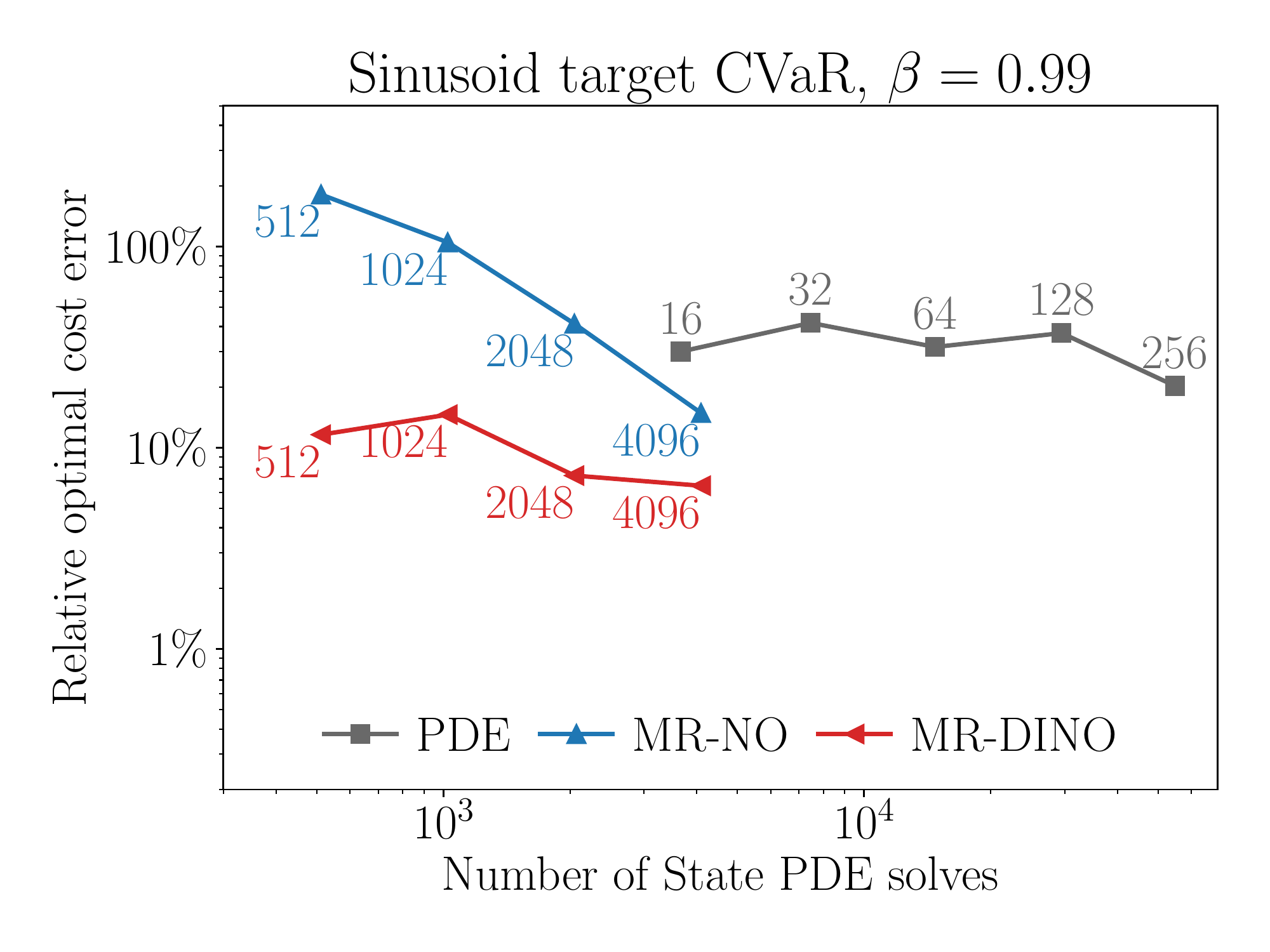}
  \end{subfigure}
  \caption{Relative error of the cost functional versus the number of state PDE solves required for the optimization with different target states and $\beta$ values.
  }
  \label{fig:ouu_poisson_cvar_error}
  \vskip -0.5cm
\end{figure}

The accuracy is more clearly compared in Figure \ref{fig:ouu_poisson_cvar_error}, where we plot the relative errors of the CVaR values at the optimal controls with respect to the reference PDE solution. 
We additionally include OUU runs where the cost functional is CVaR with quantiles $\beta = 0.9$ and $\beta = 0.99$. 
Here we observe that across all four cases, the neural operators with Jacobian training (MR-DINO) are able to obtain OUU solutions with lower cost values than the PDE-based solutions, despite using over $10 \times$ fewer state PDE solves. 
On the other hand, neural operators without Jacobian training are only more cost-effective than the PDE-based solutions across a single OUU run in the sinusoidal cases.
Evidently, Jacobian training adds a valuable source of information 
at a fraction of the cost of the state PDE solve, and is highly beneficial to the accuracy of the neural operator OUU solutions, especially in the low-data regime. 
It is important to note that in these results, the same neural operators are used to solve the OUU problems across all four cases.
The training cost can easily be amortized across different choices of the performance functions and risk measures, which in this example, correspond to different target states and $\beta$ values.
Moreover, we observe that for both the neural operator and PDE results, the optimality of the OUU solutions degrade as the quantile value $\beta$ increases. Larger $\beta$ values represent increased weighting of the tail of the distribution, which require more samples for both neural operator training and accurate estimation of the CVaR. 
Nevertheless, for the $\beta = 0.99$ case considered, the MR-DINO solutions are still more than an order of magnitude more cost-effective than the PDE solutions. 

\subsection{Optimal boundary control of flow around a bluff body} \label{section:navstok2d}
Next, we consider the boundary control of flow around a bluff body with uncertain inlet conditions governed by the steady state Navier--Stokes equation. We consider the two dimensional domain shown in Figure \ref{fig:ns2d_domain}, where the setup is analogous to that considered in \cite{ManzoniQuarteroniSalsa21}. We write the flow equations as  

\begin{subequations}
  \begin{align}
    (\mathbf{u} \cdot \nabla) \mathbf{u} + \nabla p - \nu \Delta \mathbf{u} &= 0 \qquad x \in \Omega, \label{eq:ns_start}\\
    \nabla \cdot \mathbf{u} &= 0                               \qquad x \in \Omega, \\
    \mathbf{u} - e^m \mathbf{e}_1 &= 0                                  \qquad x \in \Gamma_I, \\
    \mathbf{T}(\mathbf{u},p) \mathbf{n} &= 0                            \qquad x \in \Gamma_O, \\
    \mathbf{u} \cdot \mathbf{n} = 0, \; \mathbf{T}(\mathbf{u},p) \mathbf{n} \cdot \mathbf{t} &= 0 \qquad x \in \Gamma_W, \\
    \mathbf{u} &= 0                                            \qquad x \in \Gamma_B, \\
    \mathbf{u} - \phi(z) \mathbf{n} &= 0                       \qquad x \in \Gamma_C. \label{eq:ns_end}
  \end{align}
  \end{subequations}      


In the above equations, the state $u = (\mathbf{u},p)$ consists of the velocity and pressure fields, $\nu=0.005$ is the viscosity, $\mathbf{T} = -p \mathbf{I} + 2 \nu \; \symm(\nabla \mathbf{u})$ is the stress tensor, where $\symm(\mathbf{A}) := (\mathbf{A} + \mathbf{A}^T)/2$. 
The boundaries $\Gamma_I$, $\Gamma_O$, $\Gamma_W$, $\Gamma_B$, and $\Gamma_C$ are as labeled in Figure \ref{fig:ns2d_domain}, $\mathbf{n}$ and $\mathbf{t}$ are the unit normal and tangent vectors along the boundaries respectively.  
The inflow velocity $\mathbf{u}|_{\Gamma_I}$ is given by the trace of a 2D lognormal random field, $e^m|_{\Gamma_I} \mathbf{e}_1$, where $m \sim \cN(0, \cC)$ and $\mathbf{e}_1 = (1,0)$.
The control variables define the normal flow velocity along the sides of the bluff body $\Gamma_B$ using a cubic B-spline representation,
$\phi(z)(x)= \sum_{i=1}^{18} z_i \phi_i(x), $
where $z_i, \phi_i$ are the weights and basis functions for the upper ($i = 1, ..., 9$) and lower ($i = 10, ..., 18$) sides respectively.

\begin{figure}[!htb]
  \centering
  \includegraphics[width=0.55\linewidth]{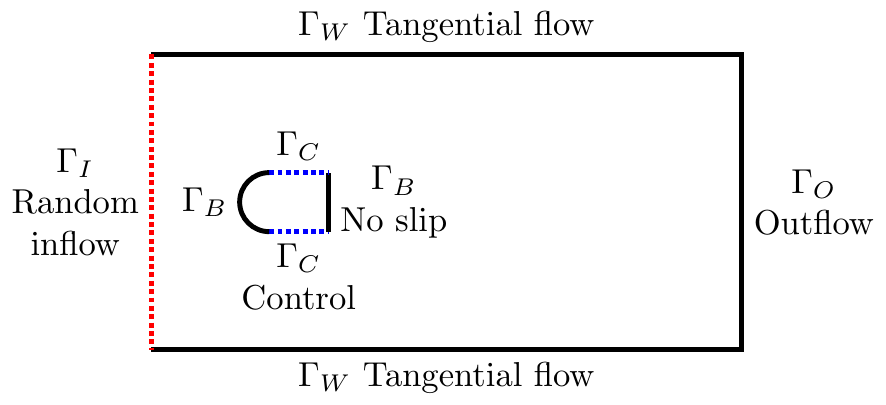}
  \caption{Flow domain $\Omega = (2,0) \times (1,0)$ with labelled boundaries.}
  \label{fig:ns2d_domain}
  \vskip -0.5cm
\end{figure}

In this problem, we consider two different performance functions. The first is the viscous dissipation rate of the flow, defined as  
\begin{equation}
  Q_{\text{Dissipation}}(\mathbf{u}) := 2 \nu \int_{\Omega} \symm(\nabla \mathbf{u}) : \symm(\nabla \mathbf{u}) dx.
\end{equation}
This corresponds to the drag force on the bluff body.
We also consider a tracking type objective, where we seek to minimize the difference between the velocity field and a target velocity field $\mathbf{u}_{\text{target}}$ behind the bluff body, 
\begin{equation}
  Q_{\text{Tracking}}(\mathbf{u}) := \int_{\Omega_o} |\mathbf{u} - \mathbf{u}_{\text{target}} |^2 dx,
\end{equation}
taking $\Omega_{o} = (0.6, 2) \times (1, 0)$ and $\mathbf{u}_{\text{target}} = (1, 0)$. 
Additionally, we adopt an $L^2$ penalization term on the velocity profile induced by the control, 
\begin{equation}
  \cP(z) = \alpha \int_{\Gamma_C} |\phi(z)|^2 ds,
\end{equation}
where $\alpha$ is a weighting parameter on the penalization term.
We consider the problem of minimizing of the CVaR with $L^2$ penalization subject to the PDE constraint, i.e.,
\begin{equation}
  \min_{z \in \cZ_{\ad}} \CVaR_{\beta}[Q](z) + \cP(z) \qquad \text{s.t. } \eqref{eq:ns_start}-\eqref{eq:ns_end}.
\end{equation}

The PDEs are discretized on a triangular mesh using Taylor--Hood elements for the state $u = (\mathbf{u},p)$ with quadratic elements for the velocity field $\mathbf{u}$ and linear elements for the pressure field $p$. This leads to the state dimension $d_U = 42,649$. Quadratic elements are also used for the random parameter field. This discretization has a nominal dimension of $d_M = 18,921$ since it is sampled on the entire domain $\Omega$, but its trace on the left boundary has 101 degrees of freedom. The state PDE is solved using a backtracking Newton method with Galerkin--Least Squares (GLS) stabilization. 

\subsubsection{Solution by neural operator}
We generate the training data for the 2D control problem using the input distribution $\nu_{m} \otimes \nu_{z}$, with a Gaussian distribution $\nu_{z} = \cN(0, I)$ as the auxiliary control distribution. 
For the reduced bases, we consider a POD basis with rank $r_U = 200$ for the state, computed from 256 samples from the training data. 
Though dimension reduction of the random inflow parameter is not necessary for this discretization, we still consider its representation in the reduced basis, since this representation is amenable to further mesh refinement. 
Here, we adopt a rank of $r_M = 100$ for the KLE basis.
We use dense neural networks with $2$ hidden layers of width $400$ to approximate the mapping from the reduced input space to the reduced output space. 
We train the neural networks using training data sets of size 256, 512, 1,024, and 2,048, both with and without Jacobian training. 
Testing errors of the trained neural operators and their Jacobians are shown in Figure \ref{fig:ns2d_training_errors}.
Similar to the previous example, we see that Jacobian training improves both the state and Jacobian accuracy of the neural operators.

\begin{figure}[!htb]
  \centering
  \begin{subfigure}{0.45\textwidth}
    \centering
    \includegraphics[width=\textwidth, trim=30 30 0 10]{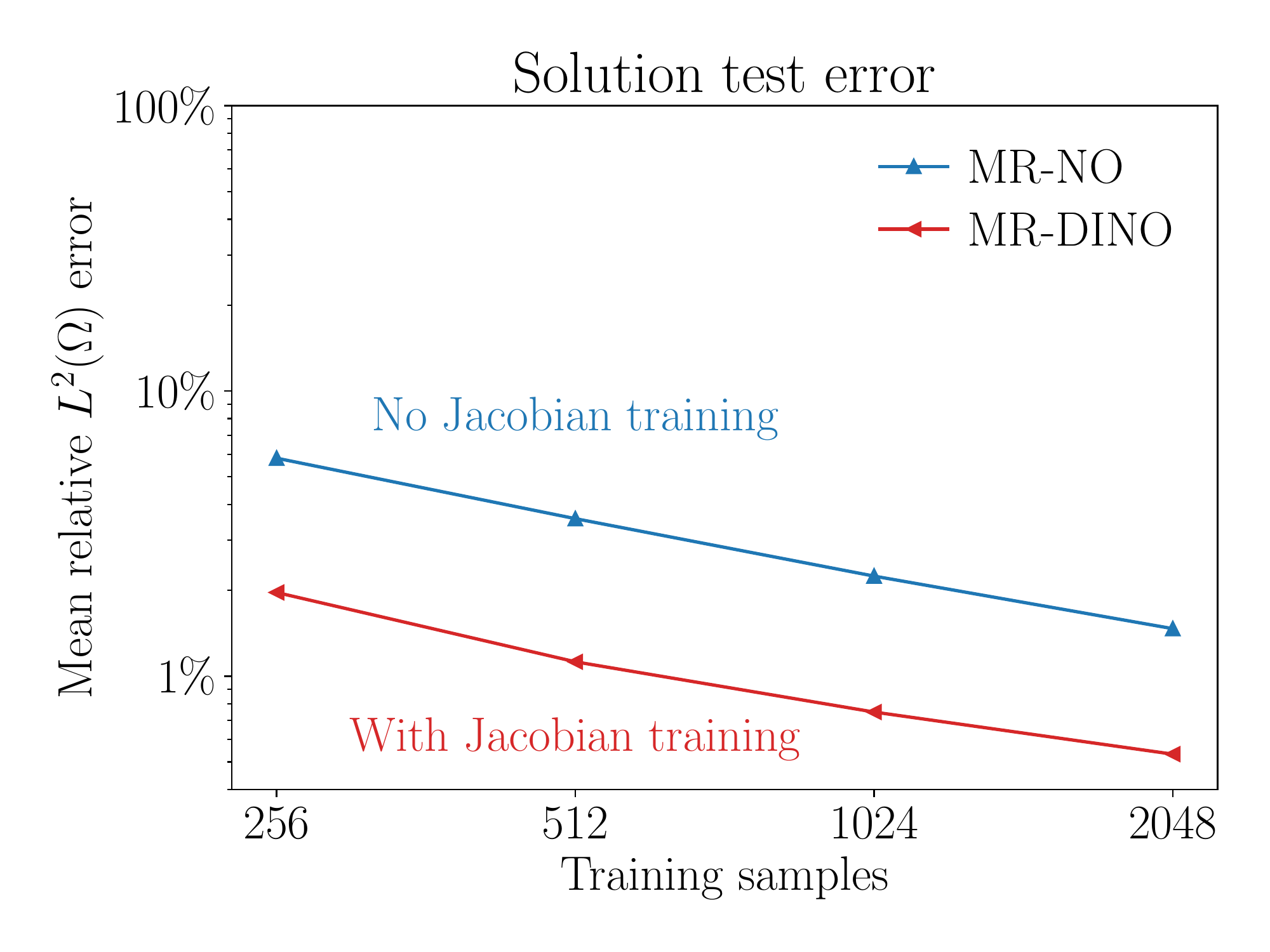}
  \end{subfigure}
  \begin{subfigure}{0.45\textwidth}
    \centering
    \includegraphics[width=\textwidth, trim=30 30 0 10]{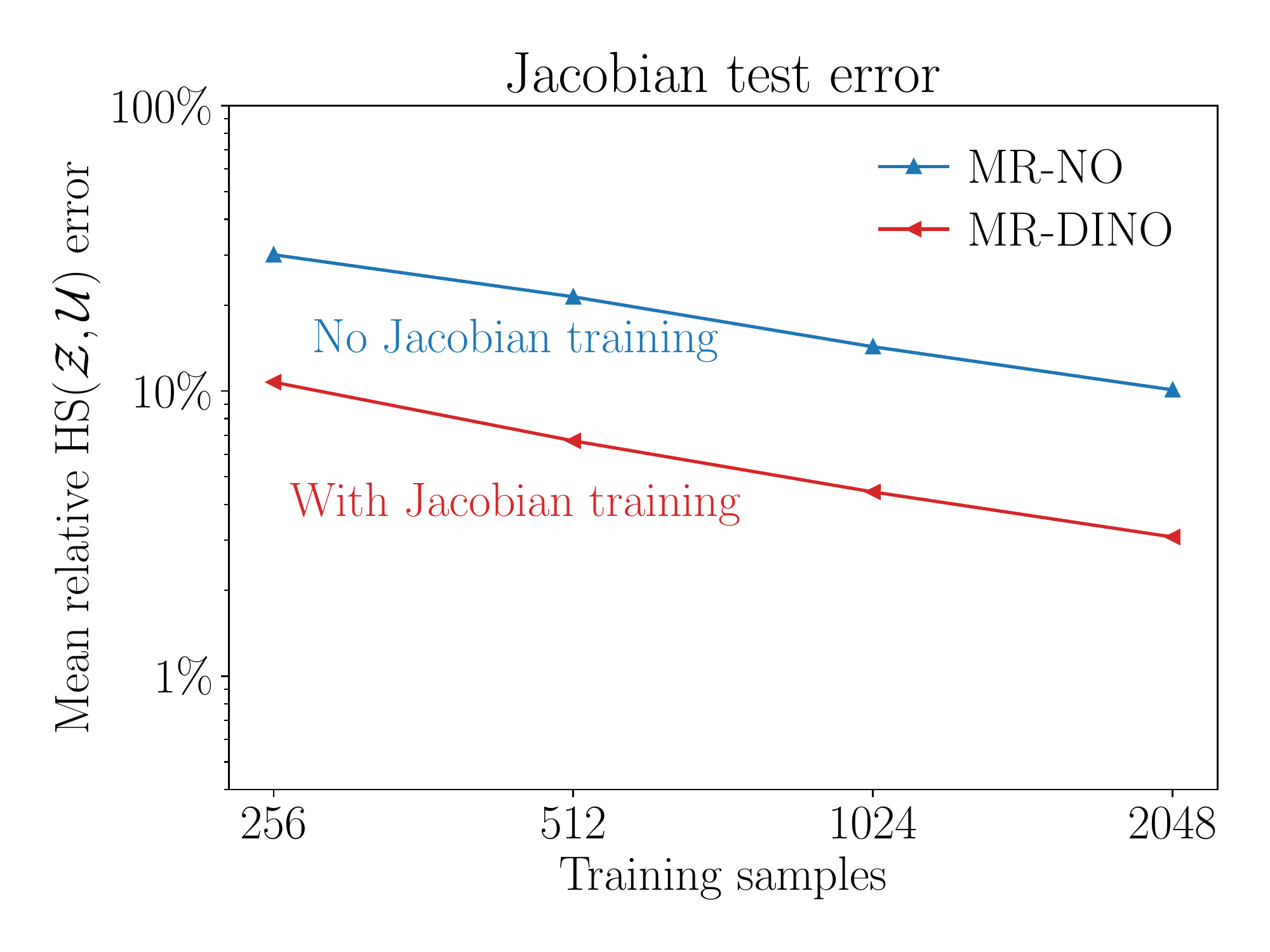}
  \end{subfigure}
  \caption{State ($L^2(\Omega)$) (left) and Jacobian ($\mathrm{HS}(\cZ, \cU)$) (right) testing errors of the 2D Navier--Stokes neural operators trained with (MR-DINO) and without Jacobian training (MR-NO). }
  \label{fig:ns2d_training_errors}
  \vskip -0.5cm
\end{figure}



We solve the OUU problem using the trained neural operator by SAA with a sample size of $N = 2,048$, and using the L-BFGS algorithm to obtain the optimal controls $z_{\NN}^{*}$. 
As an example, we consider the CVaR of the viscous dissipation rate objective with $\beta = 0.95$ and an $L^2$ penalization with $\alpha = 10^{-2}$. 
Figure \ref{fig:ns2d_controlled_uncontrolled} compares the uncontrolled flow to the controlled flow for a random inflow sample from $\nu_m$, where the controlled flow uses the optimal control computed from a MR-DINO with only 256 training samples. 
The controlled flow exhibits significantly smaller recirculation regions behind the bluff body, which effectively reduces the drag force.

\begin{figure}[!htb]
  \centering
  \begin{subfigure}{0.4\textwidth}
    \includegraphics[width=\linewidth]{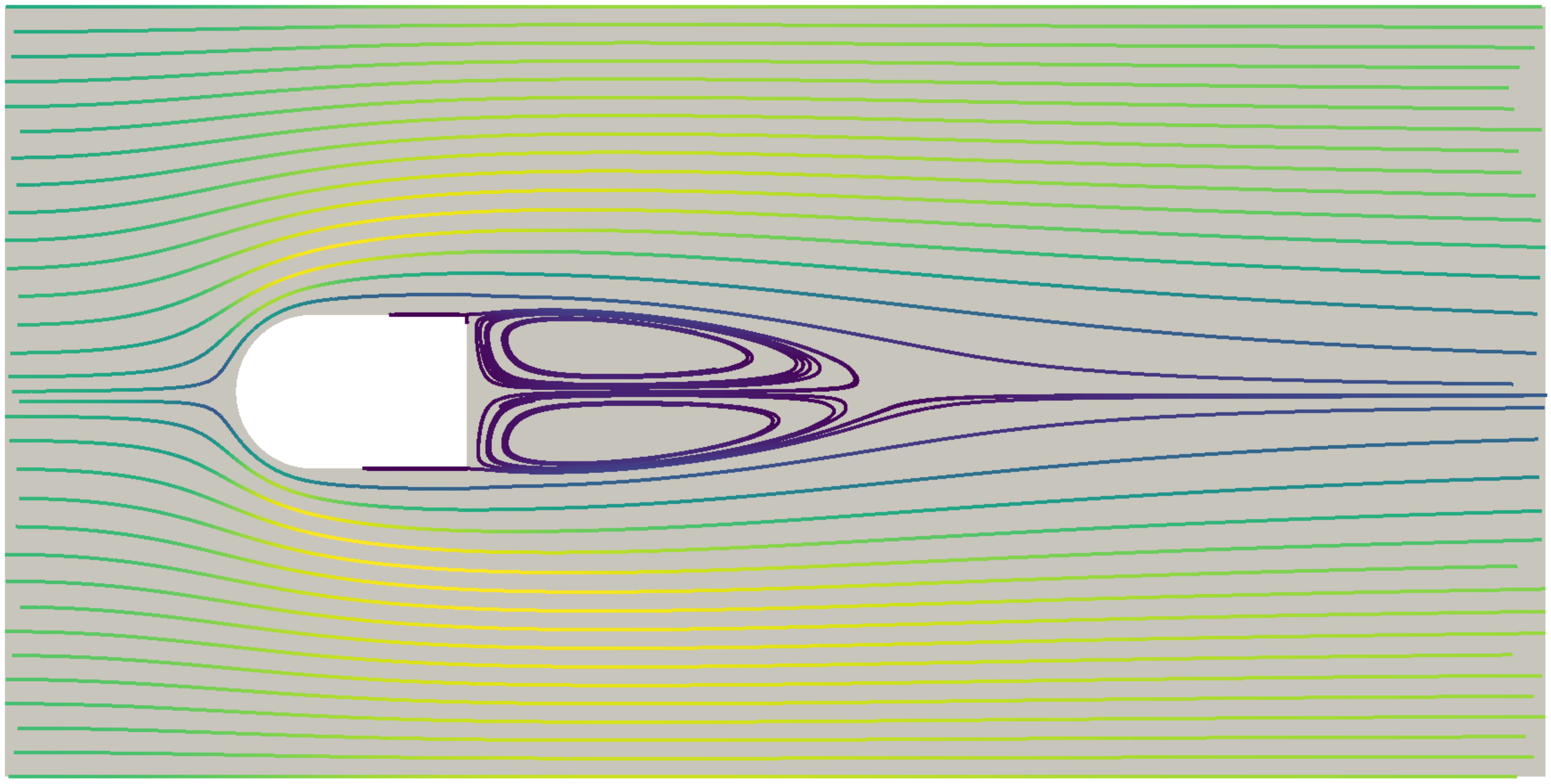}
  \end{subfigure}
  \begin{subfigure}{0.4\textwidth}
    \includegraphics[width=\linewidth]{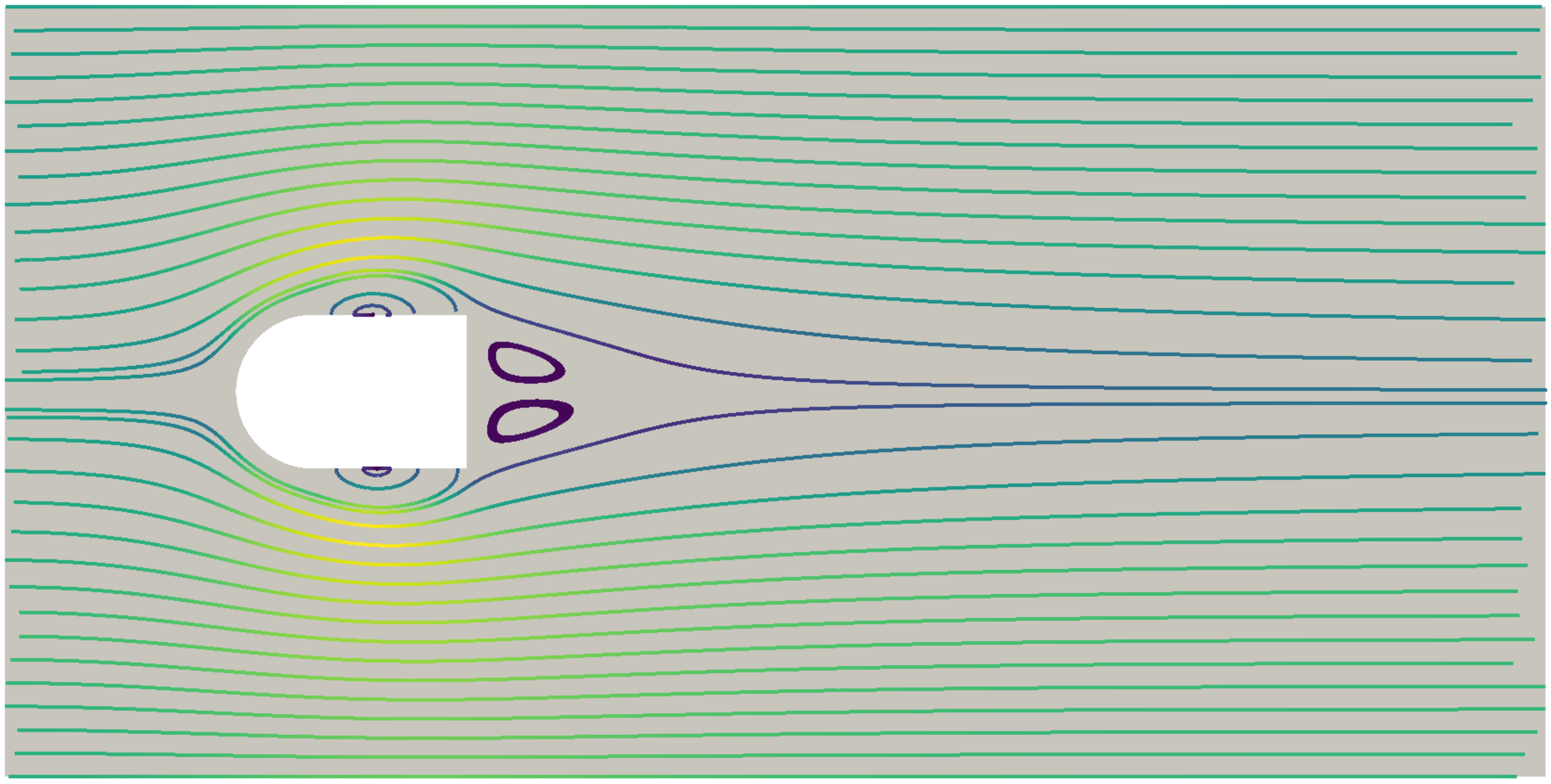}
  \end{subfigure}
  \caption{Samples of flow fields using no control (left) and the optimal control (right) computed using MR-DINO with 256 training samples. The SAA optimization problem with MR-DINO is solved in 34 seconds.}
  \label{fig:ns2d_controlled_uncontrolled}
  \vskip -0.5cm
\end{figure}

\subsubsection{Comparison with PDE solutions using SAA}
To quantify the accuracy of the neural operator, we compare the neural operator solutions against PDE solutions obtained by SAA with 16, 32, 64, and 128 samples. A reference solution is obtained by the PDE using 4,096 samples for the SAA. In addition to the viscous dissipation objective, we also consider the CVaR of the tracking objective with $\beta = 0.95$ and the penalty parameter $\alpha = 1$. We present in Figure \ref{fig:ns2d_cost_error} the relative error in the cost functional values at the optimal controls with respect to the reference PDE-based optimal solution for both the viscous dissipation and tracking objectives. As in the semilinear elliptic PDE examples, the errors are averaged across 10 runs with different random inflow samples and neural network initializations.

\begin{figure}[!htb]
  \centering
  \begin{subfigure}{0.45\textwidth}
    \includegraphics[width=\linewidth, trim=30 30 0 10]{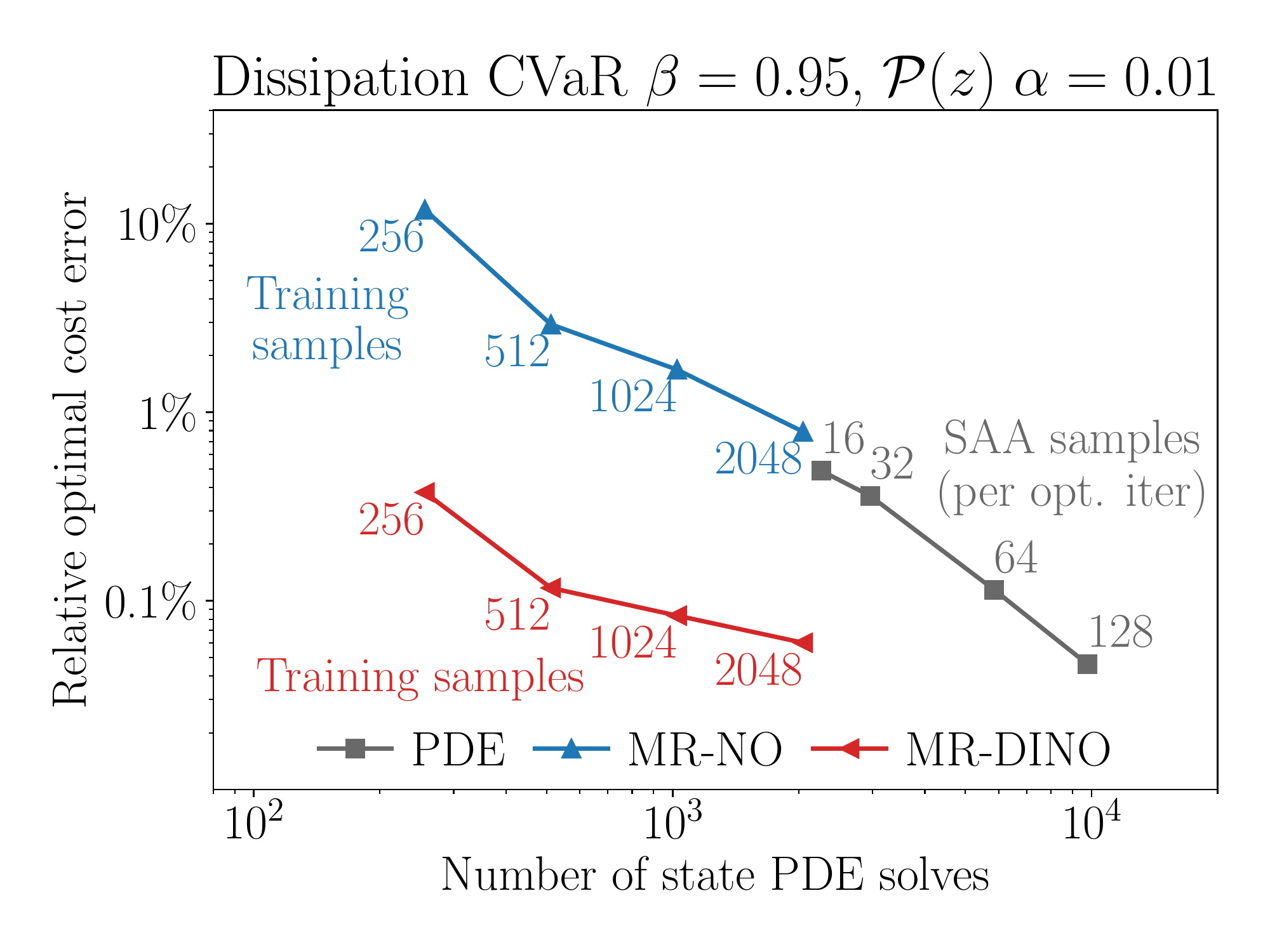}
  \end{subfigure}
  \begin{subfigure}{0.45\textwidth}
    \includegraphics[width=\linewidth, trim=30 30 0 10]{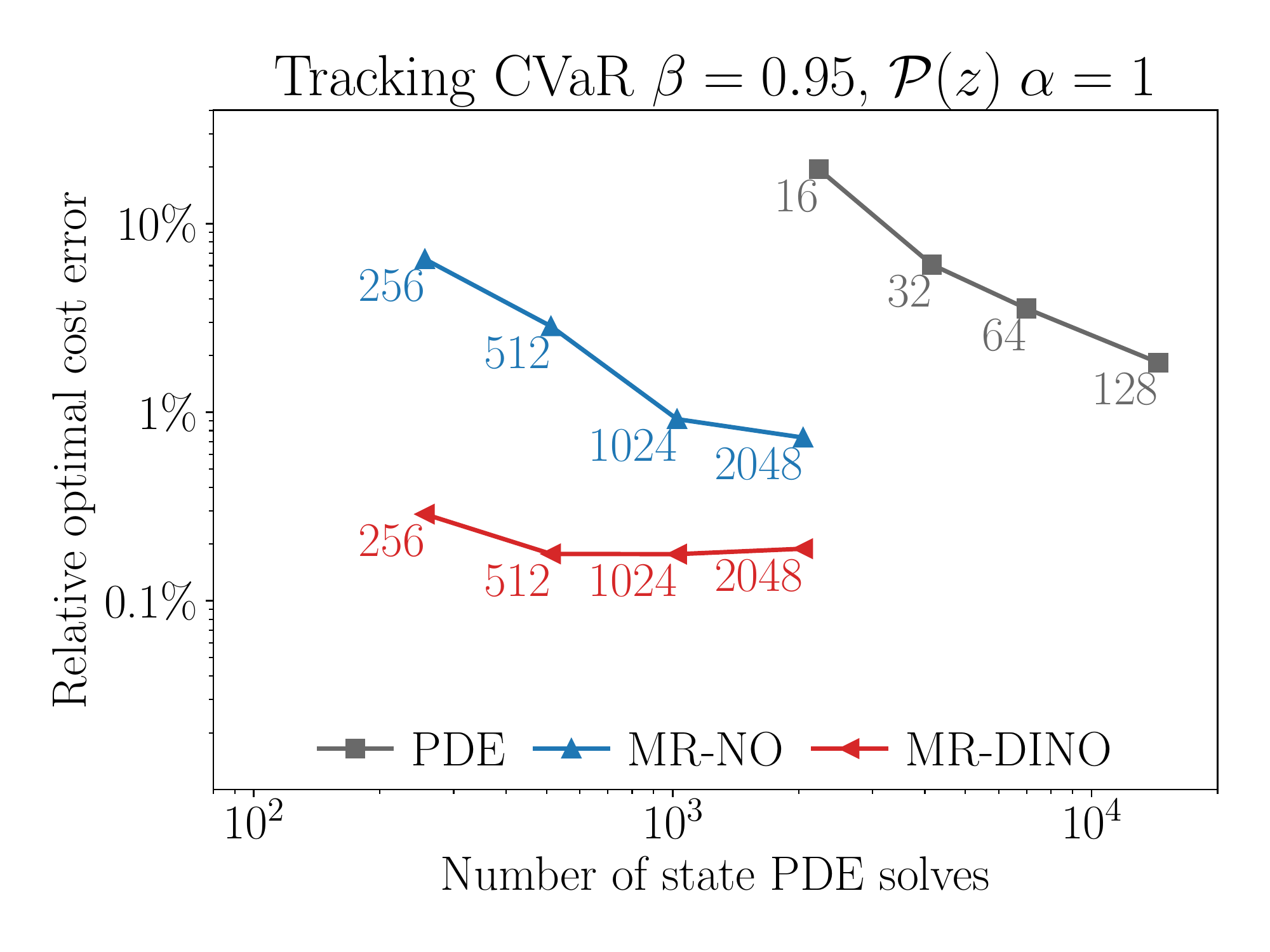}
  \end{subfigure}
  \caption{Relative error of the cost functional versus the number of state PDE solves required for the optimization with different objectives and penalty parameters. 
  }
  \label{fig:ns2d_cost_error}
  \vskip -0.5cm
\end{figure}

For the viscous dissipation objective, we observe that the PDE-based approach achieves near optimal OUU solutions with small sample sizes. 
This suggests that the SAA cost functional is well-correlated with the true CVaR cost functional, and that the OUU solution is not very sensitive to the random inflow parameter. 
Nevertheless, we observe that the neural operators with Jacobian training are able to achieve similar accuracy in optimal cost as the PDE-based approach using approximately 10 times fewer state PDE solves. In contrast, without the Jacobian training the neural operators are not as cost-effective than the PDE-based OUU solve if used only for a single optimization. 
On the other hand, the PDE-based optimal solutions for the tracking objective exhibit much larger errors compared to the reference. 
This is likely because the optimal control for the tracking objective is much more sensitive to the inflow profile. 
In this example, the neural operators with Jacobian training attain an order of magnitude smaller errors than the PDE solutions, while only using 100$\times$ fewer state PDE solves. 
Even without Jacobian training, the neural operators demonstrate a 10$\times$ improvement in cost-effectiveness compared to the PDE solutions.

We make an important observation here on the relationship between the generalization accuracy of the neural operators and their performance in solving OUU problems.
Figure \ref{fig:ns2d_training_errors} shows that the average generalization error for MR-DINO trained on 256 samples is larger than that of the MR-NO trained on 2,048 samples (without Jacobian training).
However, for the OUU problems as shown in Figure \ref{fig:ns2d_cost_error}, we see that the MR-DINO trained on 256 samples achieves lower cost values than the MR-NO trained on 2,048 samples in both the dissipation rate and tracking objectives. 
This suggests that a high $L^2$ generalization accuracy cannot guarantee the neural operator's performance in solving OUU problems, which can be corrupted by the Jacobian errors 
that are propagated to the gradients as shown in Proposition \ref{prop:gradient_error}.

\subsection{Comparison of timings} \label{section:comparison_of_timings}
First of all, to demonstrate the small incremental computational cost for generating the Jacobian data, we report the computation time for the state PDE solve, the first linearized PDE solve in Jacobian computation for which an LU factorization is computed and stored for the linear operator, and the subsequent LU solve for which only a back substitution is needed. The results are presented in Table \ref{tab:pde_timing} for the semilinear elliptic and 2D Navier--Stokes examples.

\begin{table}[!htb]
  \centering
    \begin{tabular}{|l | r | r|} 
    \hline
      Time (in seconds) & Semilinear Elliptic & 2D Navier--Stokes \\
    \hline
      State PDE solve & 0.869 & 12.81 \\
      Jacobian (LU) & 0.366 & 2.95\\
      Jacobian (Back sub.) & 0.004 & 0.02 \\
    \hline
   \end{tabular}
  \caption{Time (in seconds) for the state PDE solve, the first linearized PDE solve in Jacobian computation for which an LU factorization is computed and stored for the linear operator, and the subsequent LU solve for which only a back substitution is needed. The reported timings are averaged over 100 random samples of $(m,z)$.}
  \label{tab:pde_timing}
  \vskip -0.5cm
\end{table}

We then present in Table \ref{tab:training_times} the average training time for the neural operators with and without Jacobian training in the two examples.
The training time is broken down into three parts for (1) the training data generation (state and Jacobian training pairs), (2) preprocessing of data, and (3) neural network training.  We use the networks with two hidden layers of widths $(400, 400)$, and use the reduced dimension $(r_M, r_U) = (100, 300)$ and $(100, 200)$ for the semilinear elliptic PDE and 2D Navier--Stokes examples respectively, and use 1,024 training samples.

\begin{table}[!htb]
  \centering
    \begin{tabular}{|l | r r | r r|} 
    \hline
    Time (in seconds) & \multicolumn{2}{c|}{Semilinear Elliptic} & \multicolumn{2}{|c|}{2D Navier--Stokes} \\
        & $L^2$ Only & DINO & $L^2$ Only & DINO \\ 
    \hline
    Data Generation & 889.9 & 1,461.3 & 13,117.4 & 16,486.4\\
    Pre-processing & 2.6 & 22.8 & 29.9 & 122.5 \\
    Training & 140.7 & 650.9 & 142.0 & 486.7 \\
    \hline
    Total & 1,033.2 & 2,135.0 & 13,289.3 & 17,095.6 \\
   \hline
   \end{tabular}
   \caption{Time (in seconds) for data generation, pre-processing (reduced basis construction), and training in neural operator construction with 1,024 training samples. 
   }
   \label{tab:training_times}
   \vskip -0.5cm 
\end{table}

The semilinear elliptic PDE is mildly nonlinear, requiring 2--3 Newton iterations on average. The data generation for DINO (both state and Jacobian) takes approximately $1.5 \times$ the time of generating the state data alone. The Navier--Stokes problem is more nonlinear, requiring many more Newton iterations. The DINO data generation takes approximately $1.25 \times$ the time of generating the state data alone, meaning that the Jacobian data generation comes with only a small increase in cost.
Moreover, as the semilinear elliptic PDE solve is not expensive for the given discretization dimension, the neural network training takes a large portion of the overall time. 
On the other hand, since the neural network training is conducted in the reduced basis spaces, the Navier--Stokes example shows similar training times, despite having a state discretization that is about $10\times$ larger than the semilinear elliptic PDE example.
The increased nonlinearity and state dimension for the Navier--Stokes problem means the training costs are much lower relative to the data generation costs, being only 1--3\% of the overall time. 
In both examples, the small additional computation time for training the MR-DINOs results in large performance improvements in the OUU problem as shown in the comparisons of Sections \ref{section:semilinear_elliptic} and \ref{section:navstok2d}.

We also present the time for the evaluation of the state, the performance function $Q$, and the gradient of the performance function with respect to the control variables $z$ in Table \ref{tab:neural_operator_timing}. 
Here, we report the time for an evaluation using a single random parameter sample and that using a batch of 2,048 samples.
Since there are no architectural differences between the MR-NO and MR-DINO, we only report the timing for the MR-DINO.

\begin{table}[h!]
  \centering
   \begin{tabular}{|l | r r | r r|} 
   \hline
    Time (in seconds) & \multicolumn{2}{c|}{Semilinear Elliptic} & \multicolumn{2}{|c|}{2D Navier--Stokes} \\
        & Single & Batched (2,048) & Single & Batched (2,048) \\ 
   \hline
    State evaluation & 0.0023 & 0.10 &  0.0036 & 0.28 \\
    $Q$ evaluation & 0.0034 & 0.13 & 0.0039 & 0.27 \\
    $Q$ $z$-gradient & 0.0041 & 0.16 & 0.0045 & 0.29 \\
   \hline
   \end{tabular}
   \caption{Neural operator evaluation time for the state, the performance function, and its gradient,
 reported for a single sample and batched evaluation over 2,048 samples.}
   \label{tab:neural_operator_timing}
   \vskip -0.5cm
\end{table}

Comparing the time of the PDE solve in Table \ref{tab:pde_timing} and of the neural operator evaluation in Table \ref{tab:neural_operator_timing}, 
we see that once trained, the neural operator offers on average a $380 \times$ speed up for a single solve of the semilinear elliptic PDE and a $3,600 \times$ for a single solve of the 2D Navier--Stokes PDE. 
The speed ups are much more significant in the batched case, where the time taken to evaluate the neural operator for 2,048 different samples is a fraction of that for a single PDE solve in both examples, and $18,000 \times$ and $94,000 \times$ for the total samples. Comparisons for the performance function evaluation and gradient are similar, noting that the time required to evaluate the gradient of $Q$ using the PDE is comparable to that of a Jacobian action with LU factorization reported in Table \ref{tab:pde_timing}.



\subsection{3D Navier--Stokes example}
Finally, we demonstrate the scalability of our approach by considering a 3D Navier--Stokes example for the boundary control of the flow around a bluff body in the presence of uncertain inflow conditions, with the setup analogous to that in 2D. 
This problem is not amenable to traditional PDE-based OUU methods as a single state PDE solve takes 30 minutes with parallel computation using 48 cores in one CPU node of Frontera at TACC. 
We consider a domain of dimensions $2 \times 1 \times 1$, with a bluff body that is defined similar to the 2D case, at an angle-of-attack of 30 degrees. The inflow condition at $x = 0$ is given by $e^m \mathbf{e}_1$ with $m \sim \cN(0, (- \gamma \Delta  + \delta I)^{-2})$, where $\mathbf{e}_1 = (1,0,0)$.
We choose $\gamma = 1.5$ and $\delta = 7.5$ such that the pointwise variance and correlation lengths are similar to the 2D case. 
The control variables define the normal flow velocity along the sides of the bluff body using a tensor product of quadratic B-splines. 
Figure \ref{fig:ns3d_diagram} illustrates the obstacle geometry and the controlled region.
Using a tetrahedral mesh with Taylor--Hood elements for the state and quadratic elements for the random parameter, the discretization yields $d_U = 1,035,243$, $d_M = 4,369$ (boundary degrees of freedom), and $d_Z = 50$.

We consider the CVaR optimization problem with $\beta = 0.95$, using the viscous dissipation objective with the $L^2(\Gamma_C)$ penalization term on the control. 
To solve the OUU problem, we train a MR-DINO with 448 training samples, using input rank $r_M = 100$ and output rank $r_U = 200$. 
With Jacobian training, the neural operator is able to achieve a mean relative $L^2(\Omega)$ generalization error of $1.8 \%$. 
The neural operator is then deployed to solve the OUU problem by SAA with a sample size of $N = 1,024$, using the L-BFGS algorithm to obtain the optimal controls $z_{\NN}^{*}$. 
The neural operator takes 0.13 seconds to evaluate the performance function for all 1,024 samples, 
which is over $10^7\times$ faster than using the PDE solver.

A comparison of the flow field at a random control and the optimal control $z_{\text{NN}}^{*}$ is shown in Figure \ref{fig:ns3d_diagram} for a sample inflow profile. 
In the controlled flow field, the recirculation region is shifted towards the rear of the bluff body due to the boundary control along the bluff body, thereby reducing the overall viscous dissipation rate and hence drag. 
We remark that in the 3D problem, the boundary control is only defined in the central region of the top and bottom faces, as illustrated in Figure \ref{fig:ns3d_diagram}.
Unlike the 2D case, the flow field exhibits a more complex 3D structure, as the flow also wraps around the bluff body in the $x_3$ direction, for which the boundary control is not able to completely eliminate the recirculation region behind the bluff body.

\begin{figure}[!htb]
  \centering
  \includegraphics[width=0.55\linewidth]{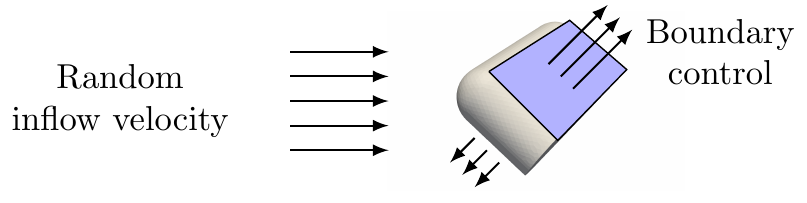}

  \includegraphics[width=0.425\linewidth, trim=80 20 80 40]{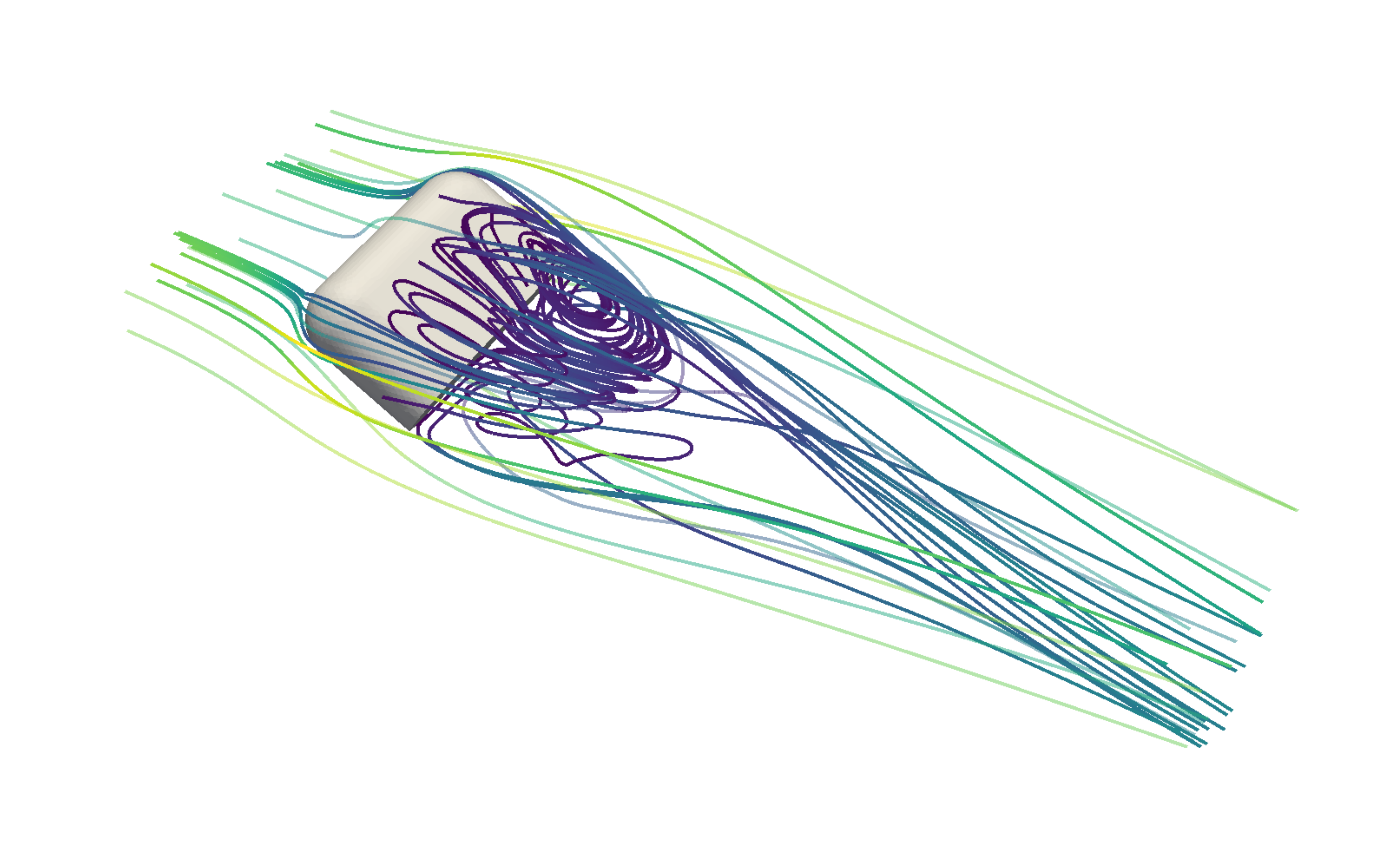}
  \includegraphics[width=0.425\linewidth, trim=80 20 80 40]{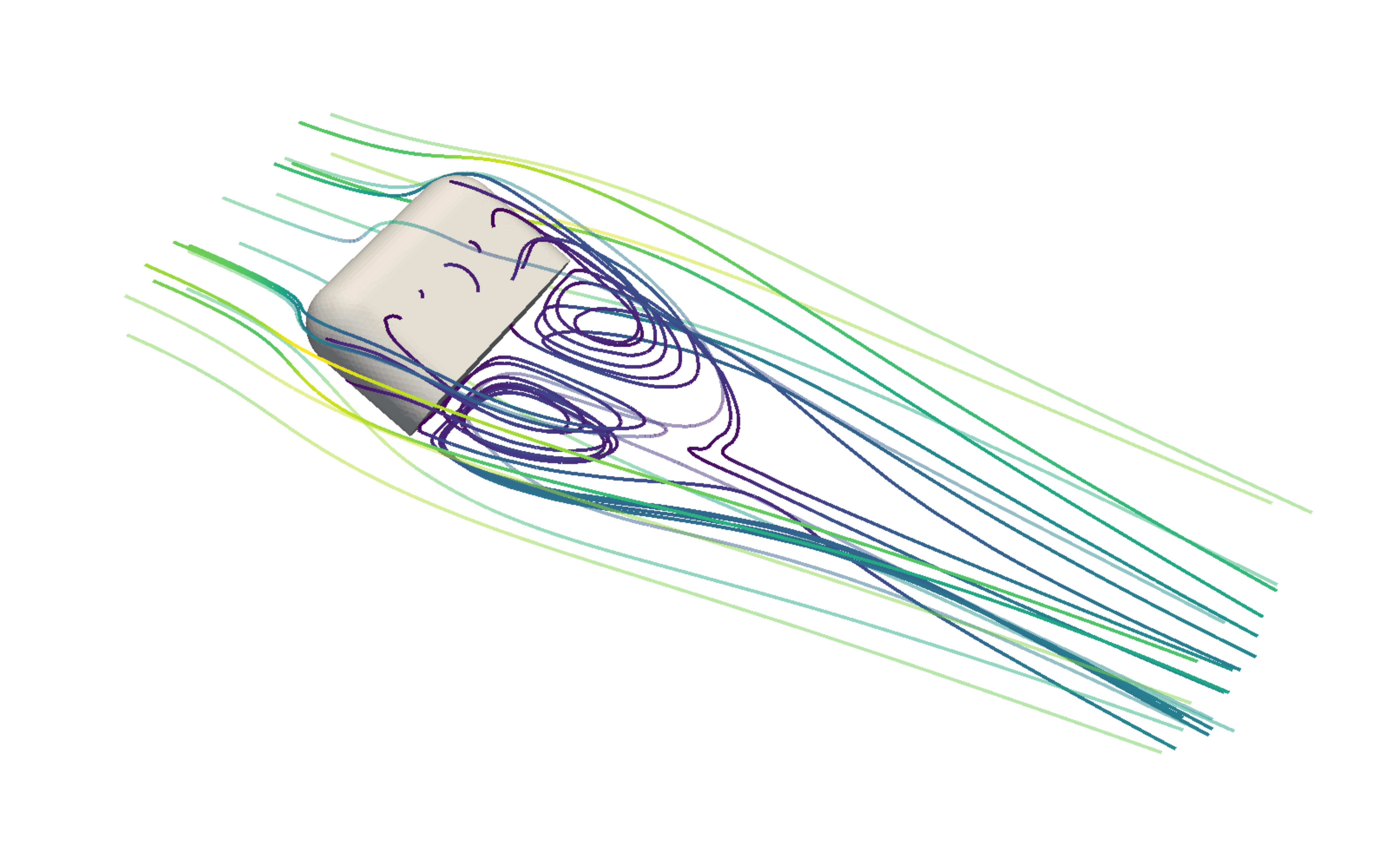}

  \caption{Top: the obstacle geometry for the 3D Navier--Stokes control problem. The control variables prescribe the normal velocity on the top and bottom faces of the obstacle over a $0.2 \times 0.4$ plane. Tangential velocity is set to zero. Bottom: Streamlines for a sample of the flow field with a random control (left) and an optimal control (right) computed using the MR-DINO with 448 training samples. The SAA optimization problem with MR-DINO is solved in 53 seconds.}
  \label{fig:ns3d_diagram}
  \vskip -0.5cm
\end{figure}



\section{Conclusions}\label{section:conclusions}

In this work, we have presented a novel framework for solving PDE-constrained OUU problems using neural operators to approximate the mapping from the joint input spaces of the uncertain parameters and optimization variables to the solution of the underlying PDE. 
The key contribution is the DINO training of neural operators on the derivatives of the solution map with respect to the optimization variable, along with the use of reduced basis architectures that enable scalable and efficient data generation and training.

Through our numerical experiments, we have demonstrated that reduced basis neural operators can be constructed to efficiently solve a range of PDE-constrained OUU problems.
In particular, we consistently showed that Jacobian training was extremely effective in improving the function approximation, the gradients, and critically, the quality of the optimization solution.
We also observed that the MR-DINOs were more cost-effective for OUU than standard SAA-based PDE solutions, with over 10$\times$ fewer state PDE solves for the same accuracy. Moreover, once trained,
the online evaluation cost with the neural operator approximation was reduced by several orders of magnitude, giving rise to the potential for real-time solution of PDE-constrained OUU problems. 
We remark that when further accuracy is required, MR-DINO can be employed in a multifidelity Monte Carlo framework \cite{NgWillcox14} to guarantee convergence to the exact OUU solution. In this setting, the low construction cost of MR-DINO relative to its accuracy makes it an appealing control variate.

Our demonstrations focused on steady-state control problems with finite dimensional optimization variables. In future work, we will apply our method to the solution of time-dependent OUU problems with function-valued optimization variables and more general risk measures and probability constraints. Furthermore, strategies for selecting the training distribution of the control variable $\nu_z$ can also be explored in future work.

\bibliographystyle{siamplain}
\bibliography{spoon,references}

\begin{thebibliography}{10}

\bibitem{tensorflow2015-whitepaper}
{\sc M.~Abadi, A.~Agarwal, P.~Barham, E.~Brevdo, Z.~Chen, C.~Citro, G.~S.
  Corrado, A.~Davis, J.~Dean, M.~Devin, S.~Ghemawat, I.~Goodfellow, A.~Harp,
  G.~Irving, M.~Isard, Y.~Jia, R.~Jozefowicz, L.~Kaiser, M.~Kudlur,
  J.~Levenberg, D.~Man\'{e}, R.~Monga, S.~Moore, D.~Murray, C.~Olah,
  M.~Schuster, J.~Shlens, B.~Steiner, I.~Sutskever, K.~Talwar, P.~Tucker,
  V.~Vanhoucke, V.~Vasudevan, F.~Vi\'{e}gas, O.~Vinyals, P.~Warden,
  M.~Wattenberg, M.~Wicke, Y.~Yu, and X.~Zheng}, {\em {TensorFlow}: Large-scale
  machine learning on heterogeneous systems}, 2015,
  \url{https://www.tensorflow.org/}.
\newblock Software available from tensorflow.org.

\bibitem{AlexanderianPetraStadlerEtAl17}
{\sc A.~Alexanderian, N.~Petra, G.~Stadler, and O.~Ghattas}, {\em Mean-variance
  risk-averse optimal control of systems governed by {PDEs} with random
  parameter fields using quadratic approximations}, SIAM/ASA Journal on
  Uncertainty Quantification, 5 (2017), pp.~1166--1192,
  \url{https://doi.org/10.1137/16M106306X}.

\bibitem{AliUllmannHinze17}
{\sc A.~A. Ali, E.~Ullmann, and M.~Hinze}, {\em Multilevel {M}onte {C}arlo
  analysis for optimal control of elliptic {PDEs} with random coefficients},
  SIAM/ASA Journal on Uncertainty Quantification, 5 (2017), pp.~466--492.

\bibitem{AllaHinzeKolvenbachEtAl19}
{\sc A.~Alla, M.~Hinze, P.~Kolvenbach, O.~Lass, and S.~Ulbrich}, {\em A
  certified model reduction approach for robust parameter optimization with pde
  constraints}, Advances in Computational Mathematics, 45 (2019),
  pp.~1221--1250.

\bibitem{BalayAbhyankarAdamsEtAl23}
{\sc S.~Balay, S.~Abhyankar, M.~F. Adams, S.~Benson, J.~Brown, P.~Brune,
  K.~Buschelman, E.~M. Constantinescu, L.~Dalcin, A.~Dener, V.~Eijkhout,
  J.~Faibussowitsch, W.~D. Gropp, V.~Hapla, T.~Isaac, P.~Jolivet, D.~Karpeev,
  D.~Kaushik, M.~G. Knepley, F.~Kong, S.~Kruger, D.~A. May, L.~C. McInnes,
  R.~T. Mills, L.~Mitchell, T.~Munson, J.~E. Roman, K.~Rupp, P.~Sanan,
  J.~Sarich, B.~F. Smith, S.~Zampini, H.~Zhang, H.~Zhang, and J.~Zhang}, {\em
  {PETS}c {W}eb page}.
\newblock \url{https://petsc.org/}, 2023, \url{https://petsc.org/}.

\bibitem{BhattacharyaHosseiniKovachkiEtAl2021}
{\sc K.~Bhattacharya, B.~Hosseini, N.~B. Kovachki, and A.~M. Stuart}, {\em
  Model reduction and neural networks for parametric {PDE}s}, SMAI Journal of
  Computational Mathematics, Volume 7,  (2021).

\bibitem{BinevCohenDahmenEtAl11}
{\sc P.~Binev, A.~Cohen, W.~Dahmen, R.~DeVore, G.~Petrova, and P.~Wojtaszczyk},
  {\em Convergence rates for greedy algorithms in reduced basis methods}, SIAM
  journal on mathematical analysis, 43 (2011), pp.~1457--1472.

\bibitem{Borzi10}
{\sc A.~Borz{\`\i}}, {\em Multigrid and sparse-grid schemes for elliptic
  control problems with random coefficients}, Computing and Visualization in
  Science, 13 (2010), pp.~153--160.

\bibitem{CaoOLearyRoseberryJhaEtAl2022}
{\sc L.~Cao, T.~O'Leary-Roseberry, P.~K. Jha, J.~T. Oden, and O.~Ghattas}, {\em
  Residual-based error correction for neural operator accelerated
  infinite-dimensional {B}ayesian inverse problems}, Journal of Computational
  Physics,  (2023), p.~112104.

\bibitem{ChaudhuriKramerNortonEtAl22}
{\sc A.~Chaudhuri, B.~Kramer, M.~Norton, J.~O. Royset, and K.~Willcox}, {\em
  Certifiable risk-based engineering design optimization}, {AIAA} Journal, 60
  (2022), pp.~551--565, \url{https://doi.org/10.2514/1.j060539}.

\bibitem{ChenGhattas21}
{\sc P.~Chen and O.~Ghattas}, {\em Taylor approximation for chance constrained
  optimization problems governed by partial differential equations with
  high-dimensional random parameters}, SIAM/ASA Journal on Uncertainty
  Quantification, 9 (2021), pp.~1381--1410.

\bibitem{ChenHabermanGhattas21}
{\sc P.~Chen, M.~Haberman, and O.~Ghattas}, {\em Optimal design of acoustic
  metamaterial cloaks under uncertainty}, Journal of Computational Physics, 431
  (2021), p.~110114.

\bibitem{ChenQuarteroni14}
{\sc P.~Chen and A.~Quarteroni}, {\em Weighted reduced basis method for
  stochastic optimal control problems with elliptic {PDE} constraints},
  SIAM/ASA J. Uncertainty Quantification, 2 (2014), pp.~364--396.

\bibitem{ChenQuarteroniRozza16}
{\sc P.~Chen, A.~Quarteroni, and G.~Rozza}, {\em Multilevel and weighted
  reduced basis method for stochastic optimal control problems constrained by
  {S}tokes equations}, Numerische Mathematik, 133 (2016), pp.~67--102.

\bibitem{ChenQuarteroniRozza17}
{\sc P.~Chen, A.~Quarteroni, and G.~Rozza}, {\em Reduced basis methods for
  uncertainty quantification}, SIAM/ASA Journal on Uncertainty Quantification,
  5 (2017), pp.~813--869.

\bibitem{ChenRoyset21}
{\sc P.~Chen and J.~O. Royset}, {\em Performance bounds for {PDE}-constrained
  optimization under uncertainty}, arXiv:2110.10269, accepted in SIAM Journal
  on Optimization,  (2023).

\bibitem{ChenSchwab15}
{\sc P.~Chen and C.~Schwab}, {\em Sparse-grid, reduced-basis {B}ayesian
  inversion}, Computer Methods in Applied Mechanics and Engineering, 297
  (2015), pp.~84 -- 115.

\bibitem{ChenSchwab16c}
{\sc P.~Chen and C.~Schwab}, {\em Model order reduction methods in
  computational uncertainty quantification}, Handbook of Uncertainty
  Quantification, Springer,  (2016).

\bibitem{ChenSchwab16}
{\sc P.~Chen and C.~Schwab}, {\em Sparse-grid, reduced-basis {B}ayesian
  inversion: Nonaffine-parametric nonlinear equations}, Journal of
  Computational Physics, 316 (2016), pp.~470--503.

\bibitem{ChenVillaGhattas19}
{\sc P.~Chen, U.~Villa, and O.~Ghattas}, {\em Taylor approximation and variance
  reduction for {PDE}-constrained optimal control under uncertainty}, Journal
  of Computational Physics, 385 (2019), pp.~163--186,
  \url{https://arxiv.org/abs/1804.04301}.

\bibitem{CohenDeVore15}
{\sc A.~Cohen and R.~DeVore}, {\em Approximation of high-dimensional parametric
  {PDEs}}, Acta Numerica, 24 (2015), pp.~1--159.

\bibitem{DuOLearyRoseberryChaudhuriEtAl2023}
{\sc X.~Du, J.~R. Martins, T.~O’Leary-Roseberry, A.~Chaudhuri, O.~Ghattas,
  and K.~E. Willcox}, {\em {Learning Optimal Aerodynamic Designs through
  Multi-Fidelity Reduced-Dimensional Neural Networks}}, in AIAA SCITECH 2023
  Forum, 2023, p.~0334.

\bibitem{EigelHaaseNeumann22}
{\sc M.~Eigel, M.~Haase, and J.~Neumann}, {\em Topology optimisation under
  uncertainties with neural networks}, Algorithms, 15 (2022), p.~241,
  \url{https://doi.org/10.3390/a15070241}.

\bibitem{FrescaManzoni2022}
{\sc S.~Fresca and A.~Manzoni}, {\em {POD-DL-ROM}: enhancing deep
  learning-based reduced order models for nonlinear parametrized {PDE}s by
  proper orthogonal decomposition}, Computer Methods in Applied Mechanics and
  Engineering, 388 (2022), p.~114181.

\bibitem{GuthSchillingsWeissmann21}
{\sc P.~A. Guth, C.~Schillings, and S.~Weissmann}, {\em A general framework for
  machine learning based optimization under uncertainty}, 2021,
  \url{https://doi.org/10.48550/arXiv.2112.11126}.

\bibitem{HaoYingSuEtAl2022}
{\sc Z.~Hao, C.~Ying, H.~Su, J.~Zhu, J.~Song, and Z.~Cheng}, {\em Bi-level
  physics-informed neural networks for {PDE} constrained optimization using
  {b}royden's hypergradients}, arXiv preprint arXiv:2209.07075,  (2022).

\bibitem{HesthavenUbbiali18}
{\sc J.~S. Hesthaven and S.~Ubbiali}, {\em Non-intrusive reduced order modeling
  of nonlinear problems using neural networks}, Journal of Computational
  Physics, 363 (2018), pp.~55--78,
  \url{https://doi.org/10.1016/j.jcp.2018.02.037}.

\bibitem{HwangLeeShinEtAl2022}
{\sc R.~Hwang, J.~Y. Lee, J.~Y. Shin, and H.~J. Hwang}, {\em Solving
  {PDE}-constrained control problems using operator learning}, in Proceedings
  of the AAAI Conference on Artificial Intelligence, vol.~36, 2022,
  pp.~4504--4512.

\bibitem{JinMengLu2022}
{\sc P.~Jin, S.~Meng, and L.~Lu}, {\em {MIONet: Learning multiple-input
  operators via tensor product}}, SIAM Journal on Scientific Computing, 44
  (2022), pp.~A3490--A3514.

\bibitem{KeilKleikampLorentzenEtAl2022}
{\sc T.~Keil, H.~Kleikamp, R.~J. Lorentzen, M.~B. Oguntola, and M.~Ohlberger},
  {\em Adaptive machine learning-based surrogate modeling to accelerate
  {PDE}-constrained optimization in enhanced oil recovery}, Advances in
  Computational Mathematics, 48 (2022), p.~73.

\bibitem{KodakkalKeithKhristenkoEtAl22}
{\sc A.~Kodakkal, B.~Keith, U.~Khristenko, A.~Apostolatos, K.-U. Bletzinger,
  B.~Wohlmuth, and R.~Wüchner}, {\em Risk-averse design of tall buildings for
  uncertain wind conditions}, Computer Methods in Applied Mechanics and
  Engineering, 402 (2022), p.~115371,
  \url{https://doi.org/10.1016/j.cma.2022.115371},
  \url{https://www.sciencedirect.com/science/article/pii/S0045782522004443}.
\newblock A Special Issue in Honor of the Lifetime Achievements of J. Tinsley
  Oden.

\bibitem{KouriHeinkenschloosVanBloemenWaanders12}
{\sc D.~Kouri, D.~Heinkenschloos, M.~Ridzal, and B.~Van Bloemen~Waanders}, {\em
  A trust-region algorithm with adaptive stochastic collocation for {PDE}
  optimization under uncertainty}, SIAM Journal on Scientific Computing, 35
  (2012), pp.~1847--1879.

\bibitem{KouriSurowiec16}
{\sc D.~P. Kouri and T.~M. Surowiec}, {\em Risk-averse {PDE}-constrained
  optimization using the conditional value-at-risk}, SIAM Journal on
  Optimization, 26 (2016), pp.~365--396,
  \url{https://doi.org/10.1137/140954556}.

\bibitem{KovachkiLiLiuEtAl2021}
{\sc N.~Kovachki, Z.~Li, B.~Liu, K.~Azizzadenesheli, K.~Bhattacharya,
  A.~Stuart, and A.~Anandkumar}, {\em Neural operator: Learning maps between
  function spaces}, arXiv preprint arXiv:2108.08481,  (2021).

\bibitem{KunothSchwab16}
{\sc A.~Kunoth and C.~Schwab}, {\em Sparse adaptive tensor {G}alerkin
  approximations of stochastic {PDE}-constrained control problems}, SIAM/ASA
  Journal on Uncertainty Quantification, 4 (2016), pp.~1034--1059.

\bibitem{LassUlbrich17}
{\sc O.~Lass and S.~Ulbrich}, {\em Model order reduction techniques with a
  posteriori error control for nonlinear robust optimization governed by
  partial differential equations}, SIAM Journal on Scientific Computing, 39
  (2017), pp.~S112--S139.

\bibitem{LeeKramer23}
{\sc D.~Lee and B.~Kramer}, {\em Bi-fidelity conditional value-at-risk
  estimation by dimensionally decomposed generalized polynomial chaos
  expansion}, Structural and Multidisciplinary Optimization, 66 (2023), p.~33,
  \url{https://doi.org/10.1007/s00158-022-03477-6},
  \url{https://doi.org/10.1007/s00158-022-03477-6}.

\bibitem{LiKovachkiAzizzadenesheliEtAl2020b}
{\sc Z.~Li, N.~Kovachki, K.~Azizzadenesheli, B.~Liu, K.~Bhattacharya,
  A.~Stuart, and A.~Anandkumar}, {\em Multipole graph neural operator for
  parametric partial differential equations}, Neural Information Processing
  Systems,  (2020).

\bibitem{LiKovachkiAzizzadenesheliEtAl2020a}
{\sc Z.~Li, N.~Kovachki, K.~Azizzadenesheli, B.~Liu, K.~Bhattacharya,
  A.~Stuart, and A.~Anandkumar}, {\em Fourier neural operator for parametric
  partial differential equations}, International Conference on Learning
  Representations,  (2021).

\bibitem{LiZhengKovachkiEtAl2021}
{\sc Z.~Li, H.~Zheng, N.~Kovachki, D.~Jin, H.~Chen, B.~Liu, K.~Azizzadenesheli,
  and A.~Anandkumar}, {\em Physics-informed neural operator for learning
  partial differential equations}, arXiv preprint arXiv:2111.03794,  (2021).

\bibitem{LindgrenRueLindstroem11}
{\sc F.~Lindgren, H.~Rue, and J.~Lindstr{\"o}m}, {\em An explicit link between
  {G}aussian fields and {G}aussian {M}arkov random fields: the stochastic
  partial differential equation approach}, Journal of the Royal Statistical
  Society: Series B (Statistical Methodology), 73 (2011), pp.~423--498,
  \url{https://doi.org/10.1111/j.1467-9868.2011.00777.x},
  \url{http://dx.doi.org/10.1111/j.1467-9868.2011.00777.x}.

\bibitem{LoggMardalGarth12}
{\sc A.~Logg, K.-A. Mardal, and G.~Wells}, {\em Automated Solution of
  Differential Equations by the Finite Element Method: {T}he {FEniCS} book},
  vol.~84, Springer Science \& Business Media, 2012.

\bibitem{LuJinKarniadakis2019}
{\sc L.~Lu, P.~Jin, G.~Pang, and G.~E. Karniadakis}, {\em Deep{ON}et:
  {L}earning nonlinear operators for identifying differential equations based
  on the universal approximation theorem of operators}, Nature Machine
  Intelligence,  (2021).

\bibitem{LuXuhuiShengzeEtAl2022}
{\sc L.~Lu, X.~Meng, S.~Cai, Z.~Mao, S.~Goswami, Z.~Zhang, and G.~E.
  Karniadakis}, {\em A comprehensive and fair comparison of two neural
  operators (with practical extensions) based on fair data}, Computer Methods
  in Applied Mechanics and Engineering, 393 (2022), p.~114778.

\bibitem{LuPestourieJohnsonEtAl22}
{\sc L.~Lu, R.~Pestourie, S.~G. Johnson, and G.~Romano}, {\em Multifidelity
  deep neural operators for efficient learning of partial differential
  equations~with application to fast inverse design of nanoscale heat
  transport}, Physical Review Research, 4 (2022), p.~023210,
  \url{https://doi.org/10.1103/physrevresearch.4.023210}.

\bibitem{LuPestourieYaoEtAl21}
{\sc L.~Lu, R.~Pestourie, W.~Yao, Z.~Wang, F.~Verdugo, and S.~G. Johnson}, {\em
  Physics-informed neural networks with hard constraints for inverse design},
  {SIAM} Journal on Scientific Computing, 43 (2021), pp.~B1105--B1132,
  \url{https://doi.org/10.1137/21m1397908}.

\bibitem{MadayPateraTurinici02}
{\sc Y.~Maday, A.~T. Patera, and G.~Turinici}, {\em A priori convergence theory
  for reduced-basis approximations of single-parameter elliptic partial
  differential equations}, Journal of Scientific Computing, 17 (2002),
  pp.~437--446.

\bibitem{ManzoniNegriQuarteroni16}
{\sc A.~Manzoni, F.~Negri, and A.~Quarteroni}, {\em Dimensionality reduction of
  parameter-dependent problems through proper orthogonal decomposition}, Annals
  of Mathematical Sciences and Applications, 1 (2016), pp.~341--377,
  \url{https://doi.org/10.4310/AMSA.2016.v1.n2.a4}.

\bibitem{ManzoniQuarteroniSalsa21}
{\sc A.~Manzoni, A.~Quarteroni, and S.~Salsa}, {\em Optimal Control of Partial
  Differential Equations}, Springer International Publishing, 2021,
  \url{https://doi.org/10.1007/978-3-030-77226-0}.

\bibitem{NelsenStuart2020}
{\sc N.~H. Nelsen and A.~M. Stuart}, {\em The random feature model for
  input-output maps between banach spaces}, SIAM Journal on Scientific
  Computing 43 (5), A3212-A3243,  (2021).

\bibitem{NgWillcox14}
{\sc L.~Ng and K.~Willcox}, {\em Multifidelity approaches for optimization
  under uncertainty}, International Journal for Numerical Methods in
  Engineering, 100 (2014), pp.~746--772,
  \url{https://doi.org/10.1002/nme.4761}.

\bibitem{OhlbergerRave15}
{\sc M.~Ohlberger and S.~Rave}, {\em Reduced basis methods: {S}uccess,
  limitations and future challenges}, arXiv preprint arXiv:1511.02021,  (2015).

\bibitem{OLearyRoseberryChenVillaEtAl22}
{\sc T.~O'Leary-Roseberry, P.~Chen, U.~Villa, and O.~Ghattas}, {\em Derivate
  informed neural operator: {A}n efficient framework for high-dimensional
  parametric derivative learning}, arXiv:2206.10745,  (2022),
  \url{http://arxiv.org/abs/2206.10745}.

\bibitem{OLeary-RoseberryDuChaudhuriEtAl22}
{\sc T.~O'Leary-Roseberry, X.~Du, A.~Chaudhuri, J.~Martins, K.~Willcox, and
  O.~Ghattas}, {\em Learning high-dimensional parametric maps via reduced basis
  adaptive residual networks}, Computer Methods in Applied Mechanics and
  Engineering, 402 (2022), p.~115730.

\bibitem{hippyflow}
{\sc T.~O'Leary-Roseberry and U.~Villa}, {\em {hIPPYflow}: {D}imension reduced
  surrogate construction for parametric {PDE} maps in {P}ython}, 2021,
  \url{https://doi.org/10.5281/zenodo.4608729},
  \url{https://github.com/hippylib/hippyflow}.

\bibitem{OLeary-RoseberryVillaChenEtAl22}
{\sc T.~O’Leary-Roseberry, U.~Villa, P.~Chen, and O.~Ghattas}, {\em
  Derivative-informed projected neural networks for high-dimensional parametric
  maps governed by {PDEs}}, Computer Methods in Applied Mechanics and
  Engineering, 388 (2022), p.~114199.

\bibitem{RaissiPerdikarisKarniadakis2019}
{\sc M.~Raissi, P.~Perdikaris, and G.~E. Karniadakis}, {\em Physics-informed
  neural networks: A deep learning framework for solving forward and inverse
  problems involving nonlinear partial differential equations}, Journal of
  Computational Physics, 378 (2019), pp.~686--707.

\bibitem{RockafellarRoyset10}
{\sc R.~Rockafellar and J.~Royset}, {\em On buffered failure probability in
  design and optimization of structures}, Reliability Engineering {\&} System
  Safety, 95 (2010), pp.~499--510,
  \url{https://doi.org/10.1016/j.ress.2010.01.001}.

\bibitem{RockafellarUryasev00}
{\sc R.~T. Rockafellar and S.~Uryasev}, {\em Optimization of conditional
  value-at-risk}, Journal of risk, 2 (2000), pp.~21--42.

\bibitem{ShapiroDentchevaRuszczynski21}
{\sc A.~Shapiro, D.~Dentcheva, and A.~Ruszczynski}, {\em Lectures on Stochastic
  Programming: Modeling and Theory}, Society for Industrial and Applied
  Mathematics, third~ed., July 2021,
  \url{https://doi.org/10.1137/1.9781611976595}.

\bibitem{ShuklaOommenPeyvanEtAl23}
{\sc K.~Shukla, V.~Oommen, A.~Peyvan, M.~Penwarden, L.~Bravo, A.~Ghoshal, R.~M.
  Kirby, and G.~E. Karniadakis}, {\em Deep neural operators can serve as
  accurate surrogates for shape optimization: A case study for airfoils}, 2023,
  \url{https://doi.org/10.48550/ARXIV.2302.00807}.

\bibitem{TieslerKirbyXiuEtAl12}
{\sc H.~Tiesler, R.~M. Kirby, D.~Xiu, and T.~Preusser}, {\em Stochastic
  collocation for optimal control problems with stochastic {PDE} constraints},
  SIAM Journal on Control and Optimization, 50 (2012), pp.~2659--2682.

\bibitem{VillaPetraGhattas21}
{\sc U.~Villa, N.~Petra, and O.~Ghattas}, {\em {HIPPYlib: An Extensible
  Software Framework for Large-Scale Inverse Problems Governed by PDEs: Part I:
  Deterministic Inversion and Linearized Bayesian Inference}}, ACM Trans. Math.
  Softw., 47 (2021), \url{https://doi.org/10.1145/3428447},
  \url{https://doi.org/10.1145/3428447}.

\bibitem{WangBhouriPerdikarisEtAl2021}
{\sc S.~Wang, M.~A. Bhouri, and P.~Perdikaris}, {\em Fast {PDE}-constrained
  optimization via self-supervised operator learning}, arXiv preprint
  arXiv:2110.13297,  (2021).

\bibitem{WuOLearyRoseberryChenEtAl2023}
{\sc K.~Wu, T.~O’Leary-Roseberry, P.~Chen, and O.~Ghattas}, {\em {Large-Scale
  Bayesian Optimal Experimental Design with Derivative-Informed Projected
  Neural Network}}, Journal of Scientific Computing, 95 (2023), p.~30.

\bibitem{YuLuMengEtAl2022}
{\sc J.~Yu, L.~Lu, X.~Meng, and G.~E. Karniadakis}, {\em Gradient-enhanced
  physics-informed neural networks for forward and inverse {PDE} problems},
  Computer Methods in Applied Mechanics and Engineering, 393 (2022), p.~114823.

\bibitem{ZahrCarlbergKouri19}
{\sc M.~J. Zahr, K.~T. Carlberg, and D.~P. Kouri}, {\em An efficient, globally
  convergent method for optimization under uncertainty using adaptive model
  reduction and sparse grids}, SIAM/ASA Journal on Uncertainty Quantification,
  7 (2019), pp.~877--912.

\bibitem{ZhaoLindellWetzstein22}
{\sc Q.~Zhao, D.~B. Lindell, and G.~Wetzstein}, {\em Learning to solve
  {PDE}-constrained inverse problems with graph networks}, 2022.

\end{thebibliography}

\end{document}